\documentclass[12pt]{article}

\usepackage{amssymb}
\usepackage{amsthm}
\usepackage{amsmath}
\usepackage{graphicx}
\usepackage{fullpage}
\usepackage{subcaption}

\usepackage{color}
\usepackage{enumerate}
 \numberwithin{equation}{section}
\usepackage{booktabs}
\allowdisplaybreaks
\newtheorem{innercustomthm}{Theorem}
\newenvironment{customthm}[1]
  {\renewcommand\theinnercustomthm{#1}\innercustomthm}
  {\endinnercustomthm}
\usepackage[pdftex]{hyperref}
 \usepackage{mathtools}
 \mathtoolsset{showonlyrefs}
 \usepackage[nocompress]{cite}

\theoremstyle{plain}
\newtheorem{thm}{Theorem}[section]
\newtheorem{cor}[thm]{Corollary}

\newtheorem{lem}[thm]{Lemma}
\newtheorem{prop}[thm]{Proposition}

\theoremstyle{definition}

\newtheorem{ex}[thm]{Example}

\theoremstyle{remark}

\newtheorem{rem}[thm]{Remark}

\newcommand{\R}{\mathbb{R}}

\newcommand{\bp}{\begin{proof}[\ensuremath{\mathbf{Proof}}]}
\newcommand{\bs}{\begin{proof}[\ensuremath{\mathbf{Solution}}]}
\newcommand{\ep}{\end{proof}}
\newcommand{\be}{\begin{equation}}
\newcommand{\ee}{\end{equation}}
\newcommand{\Stwo}{\mathbb{S}^2}
\newcommand{\Snmo}{\mathbb{S}^{n-1}}
\newcommand{\id}{\text{id}}

\begin{document}

\title{On extreme constant width bodies in $\R^3$}

\author{Ryan Hynd\footnote{Department of Mathematics, University of Pennsylvania. This work was supported in part by NSF grant DMS-2350454.}}

\maketitle

\begin{abstract}
We consider the family of constant width bodies in $\R^3$ which is convex under Minkowski addition. 
Extreme shapes cannot be expressed as a nontrivial convex combination of other 
constant width bodies. We show that each Meissner polyhedra is extreme.  We also explain that each constant width 
body obtained by rotating a symmetric Reuleaux polygon about its axis of symmetry is extreme. In addition, we 
conjecture a general characterization of all extreme constant width shapes.  
\end{abstract}

\section{Introduction}
A convex body in Euclidean space has {\it constant width} if the distance between any pair of parallel supporting planes is the same. In what follows, we will only consider convex bodies with constant width equal to one and refer to them simply as constant width bodies or constant width shapes.  We will also identify a planar constant width shape with its boundary curve, which we will call a constant width curve. Simple examples include a circle of radius one half in the plane and a closed ball of radius one half in $\R^3$.  However, there are many other examples as we will see below. 
\begin{figure}[h]
\centering
   \includegraphics[width=.31\textwidth]{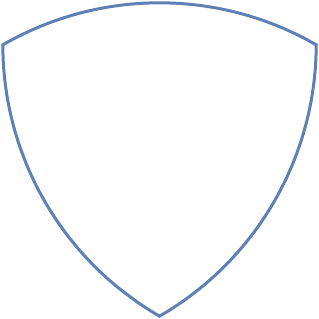}
 \hspace{.1in}
  \includegraphics[width=.31\textwidth]{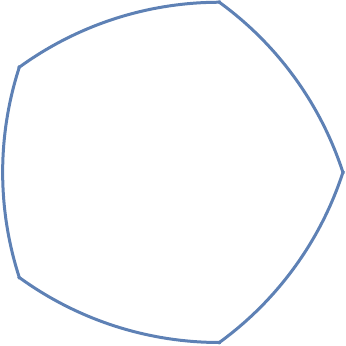}
  \hspace{.1in}
    \includegraphics[width=.31\textwidth]{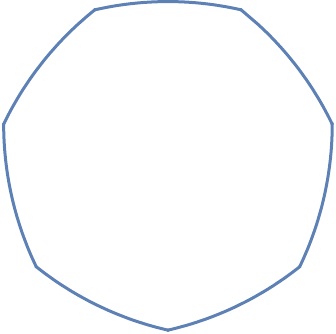}
 \caption{These are Reuleaux polygons, which are constant width curves consisting of finitely many circular arcs of radius one. It turns out that each Reuleaux polygon is extreme. In particular, each one satisfies Kallay's criterion: its radius of curvature is equal to one at circular boundary points and is equal to zero at vertex points.}
 \label{2Dcurves}
\end{figure}

\par It turns out that $(1-\lambda)K_0+\lambda K_1 $
has constant width for each $\lambda\in [0,1]$ and constant width $K_0, K_1\subset \R^n$. That is, 
the collection of constant width shapes in $\R^n$ form a convex set under Minkowski addition. We will say that a constant width 
$K\subset \R^n$ is {\it extreme} if   
$$
K\neq (1-\lambda)K_0+\lambda K_1
$$
for any $\lambda\in (0,1)$ and constant width $K_0, K_1\subset \R^n$ which are not translates of each other.  

\par When $n=2$, Kallay gave a necessary and sufficient condition for a constant width curve to be extreme. His characterization is based on the fact that the radius of curvature of a smooth constant width curve is between zero and one at each boundary point.  Kallay showed that a constant width curve is extreme if and only if its radius of curvature is essentially either zero or one \cite{MR350618}. A simple example of an extreme constant width curve is a Reuleaux polygon, which is a constant width curve consisting of finitely many circular arcs of radius one; see Figure \ref{2Dcurves}. Unfortunately, no such characterization is currently available for $n\ge 3$. 

\par In this note, we will discuss extreme constant width shapes in $\R^3$.  The first class that we will study is the  family of Meissner polyhedra, which was introduced by Montejano and Rold\'an-Pensado \cite{MR3620844}.  The family of Meissner polyhedra include the two Meissner tetrahedra (Figures \ref{MeissTetraO} and \ref{MeissTetraT}), which are conjectured to have least volume among all three-dimensional shapes of constant width \cite{MR920366, MR2844102,MR3930585,MR0123962}.  We will define this class of figures in the following section. We note that  Moreno and Schneider have shown that the two Meissner tetrahedra are extreme \cite{MR2946457}.  One of the contributions of this work is to extend their result to all Meissner polyhedra. 
 
 \begin{customthm}{A}\label{thmA}
Each Meissner polyhedron in $\R^3$ is extreme. 
\end{customthm}
 \begin{figure}[h]
\centering
 \includegraphics[width=.4\textwidth]{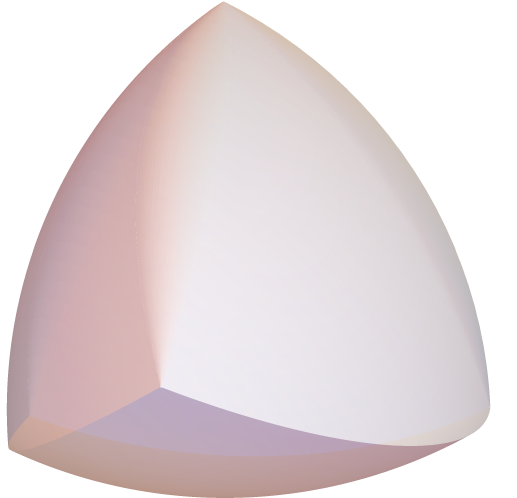}
 \hspace{.4in}
  \includegraphics[width=.38\textwidth]{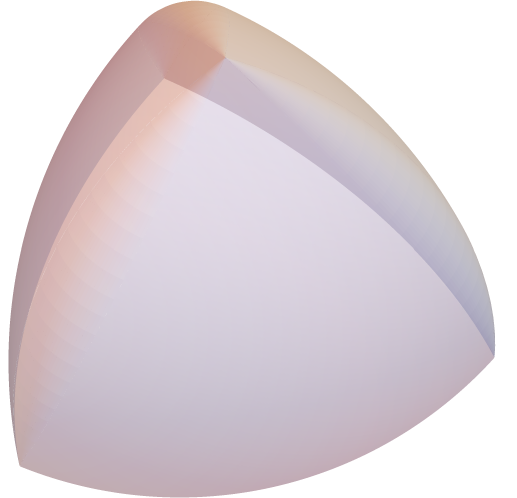}
 \caption{A Meissner tetrahedron, which is an extreme constant width shape in $\R^3$. This shape is designed from a Reuleaux tetrahedron by performing surgery on  three edges that meet in a common vertex. See Figure \ref{ReulTetraFig} below. }\label{MeissTetraO}
\end{figure}
\begin{figure}[h]
\centering
 \includegraphics[width=.45\textwidth]{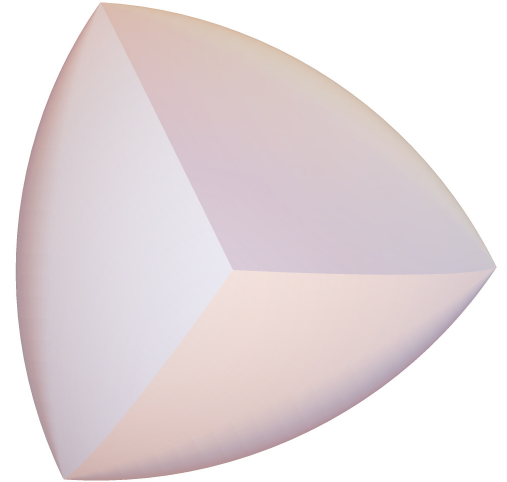}
 \hspace{.4in}
  \includegraphics[width=.43\textwidth]{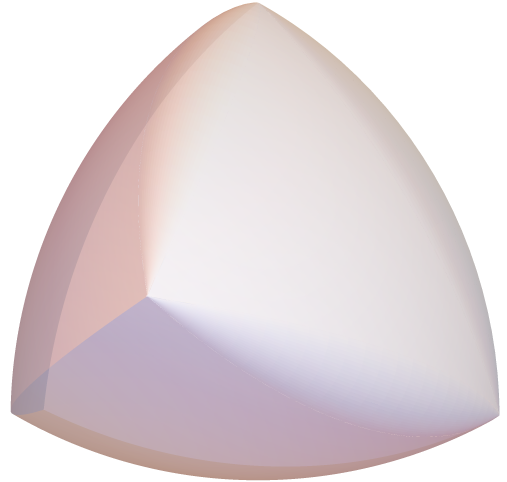}
 \caption{This is another type of Meissner tetrahedron, which is also extreme.  It is constructed from a Reuleaux tetrahedron (Figure \ref{ReulTetraFig}) by performing surgery on three edges which meet in a common face of the Reuleaux tetrahedron. }
 \label{MeissTetraT}
\end{figure}

\par The family of Meissner polyhedra was recently shown to be dense in the space of constant width shapes \cite{MR4775724}; see also \cite{MR296813}.  Here the topology is determined by the Hausdorff distance.   As a result, the space of three-dimensional constant width shapes has a dense collection of  extreme points.  This also holds in plane as Reuleaux polygons have this property.  A first sight, this property might appear to be special. However, it is not out of the ordinary for a closed convex subset of a Banach space to be the closure of its extreme points (see section 2 of \cite{MR115076}).  

 \par A constant width $C\subset \R^2$ which has a line of symmetry can be rotated about this line to generate a constant width body in $K\subset \R^3$.  In this case, we'll say that $C$ {\it generates} $K$. For instance, we can rotate a symmetric Reuleaux polygon about its line of symmetry to obtain a constant width shape in $\R^3$; see Figure \ref{RotatedFigFig}. Conversely, an axially symmetric constant width shape arises this way, as well.  We will show that any axially symmetric constant width shape generated by a Reuleaux polygon is extreme. 
\begin{figure}[h]
\centering
 \includegraphics[width=.4\textwidth]{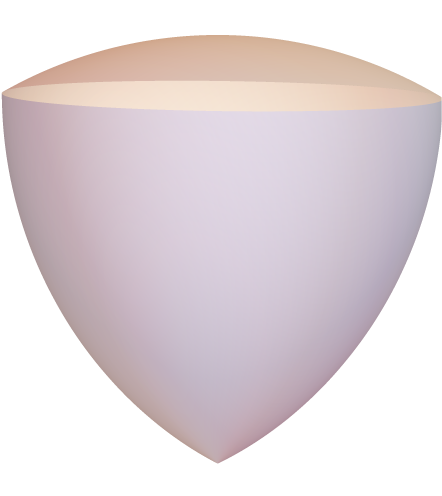}
 \hspace{.4in}
  \includegraphics[width=.4\textwidth]{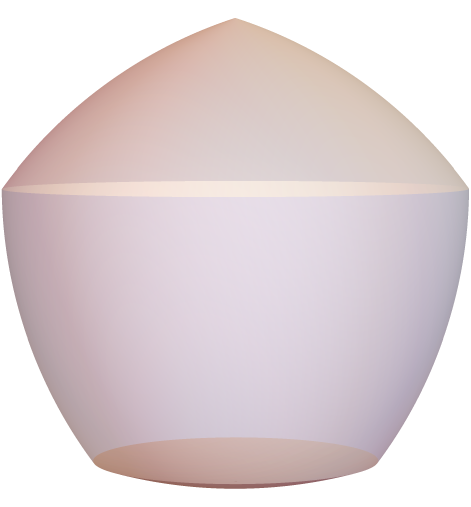}
 \caption{This figure displays a rotated Reuleaux triangle on the left and a rotated Reuleaux pentagon on the right. These are two examples of extreme axially symmetric constant width shapes in $\R^3$.}
 \label{RotatedFigFig}
\end{figure}
 \begin{customthm}{B}\label{thmB}
Each axially symmetric constant width shape in $\R^3$ which is generated by a Reuleaux polygon is extreme.
\end{customthm}
\par  We will prove the theorems above using intersection properties of constant width shapes. Then we will adapt the approach of Kallay who gave his characterization of two dimensional extreme sets in terms of their support function.  This is a function that indicates the location of the supporting halfspaces of a given convex body.  It will be our tool in strengthening Theorem \ref{thmB}. In particular, we will establish the following. 
\begin{customthm}{C}\label{thmC}
An axially symmetric constant width shape in $\R^3$ is extreme if and only if its generating shape in $\R^2$ is extreme. 
\end{customthm}

\par As with constant width curves, the constant width condition for a smooth constant width $K\subset\R^3$ requires that the principal principal radii of curvatures are between zero and one at each boundary point. In general, the boundary of a constant width $K\subset \R^3$ may have singularities. Nevertheless,  we can use the support function to interpret radii of curvature for almost every outward unit normal and these functions are bounded between zero and one.  Here ``almost every" means with respect to the standard spherical measure. We conjecture the following characterization of extreme constant width shapes in analogy to Kallay's characterization of constant width curves. 
\\\\
\noindent {\bf Conjecture.} {\it A constant width body in $\R^3$ is extreme if and only its minimum principal radius of curvature is equal to zero or its maximum principal radius of curvature is equal one for almost every outward unit normal.}
\\
\par This paper is divided into several sections. In the following two sections, we will prove Theorems \ref{thmA} and \ref{thmB}. Next we will take a short detour to recall some useful facts about the support function of a constant width shape. Then we will prove Theorem  \ref{thmC} and collect some supporting evidence for the above conjecture in the final two sections. 
\section{Meissner polyhedra}
 Suppose $X\subset\R^3$ is a finite set with at least four points and whose diameter is equal to one.  We will consider the associated {\it ball polyhedron}
$$
B(X):=\bigcap_{x\in X} B(x).
$$
Here $B(x)$ is the closed ball of radius one centered at $x\in X$. We will also use the above definition of $B(X)$ for general subsets $X$ of $\R^3$ below.  Let us further assume that $X$ is tight in the sense that no $x\in X$ can be removed so that $B(X)=B(X\setminus\{x\})$. 
The boundary of $B(X)$ is 
$$
\partial B(X)=\bigcup_{x\in X}\partial B(x)\cap B(X)
$$
This boundary naturally consists of vertices, circular edges, and spherical faces which we will describe in more detail below.  The references for this material are \cite{MR2593321, MR2343304, MR3930585}. 

\par A {\it face} of $ B(X)$ is $\partial B(x)\cap B(X)$ for a given $x\in X$. It turns out there is exactly one face per $x\in X$.  A {\it principal vertex} of $ B(X)$ is a point $y\in B(X)$ which belongs 
to three or more faces of $ B(X)$.  A {\it dangling vertex} of $ B(X)$ is a point $x\in X$ which belongs to exactly two faces of $ B(X)$. We will denote the set of principal and dangling vertices of $ B(X)$ 
as $\text{vert}(B(X))$. 

\par Recall that for $a,b\in \R^3$ with $|a-b|\le 2$, $\partial B(a)\cap \partial B(b)$ consists of all $z\in \R^3$ with 
$$
\displaystyle\left(z-\frac{a+b}{2}\right)\cdot (a-b)=0\quad \text{and}\quad \displaystyle\left|z-\frac{a+b}{2}\right|=\sqrt{1-\left|\frac{a-b}{2}\right|^2}.
$$
In particular, $\partial B(a)\cap \partial B(b)$ is a circle.  The {\it edges} of $B(X)$ are the connected components of $\partial B(x)\cap \partial B(y)\cap B(X)\setminus X$ as $x$ and $y$ range over 
distinct points of $X$.  It is known that the number of vertices $V$, edges $E$, and faces $F$ of $B(X)$ satisfy the Euler relation 
$$
V-E+F=2.
$$
\par We'll say that a pair $\{x,y\}\subset X$ is diametric if $|x-y|=1$. A seminal theorem due independently to Gr\"unbaum \cite{MR87115}, Heppes \cite{MR87116}, and Straszewicz \cite{MR0087117} is that $X$ has at most $2\#X-2$ diametric pairs.  When $X$ has $2\#X-2$ diametric pairs,  we will say that $X$ is {\it extremal}.  A fundamental result of Kupitz, Martini, and Perles asserts that $X$ is extremal if and only if 
$$
X=\text{vert}(B(X))
$$
\cite{MR2593321}.  Therefore, $X$ is extremal if and only if the centers which determine the ball polyhedron $B(X)$ are also the vertices of this ball polyhedron. 

\par The Euler relation above implies that if $\#X=m$, then $B(X)$ has $2m-2=2(m-1)$ edges.  It turns out that the edges of $B(X)$ are naturally grouped into pairs. In particular, for an edge $$e\in \partial B(x)\cap \partial B(y)\cap B(X)$$ with endpoints $x',y'\in X$, there is a unique {\it dual edge} $$e'\in \partial B(x')\cap \partial B(y)\cap B(X)$$ with endpoints $x,y\in X$.

\par When $X$ is extremal, we will call the corresponding ball polyhedron
a {\it Reuleaux polyhedron}.   The simplest example of an extremal set is the set of vertices $\{x_1, x_2, x_3, x_4\}$ of a regular tetrahedron in $\R^3$ of side length one.  
The corresponding Reuleaux polyhedron is known as a Reuleaux tetrahedron since it has four vertices, six edges, and four faces like a regular tetrahedron in $\R^3$. See Figure \ref{ReulTetraFig}.
 \begin{figure}[h]
\centering
 \includegraphics[width=.41\textwidth]{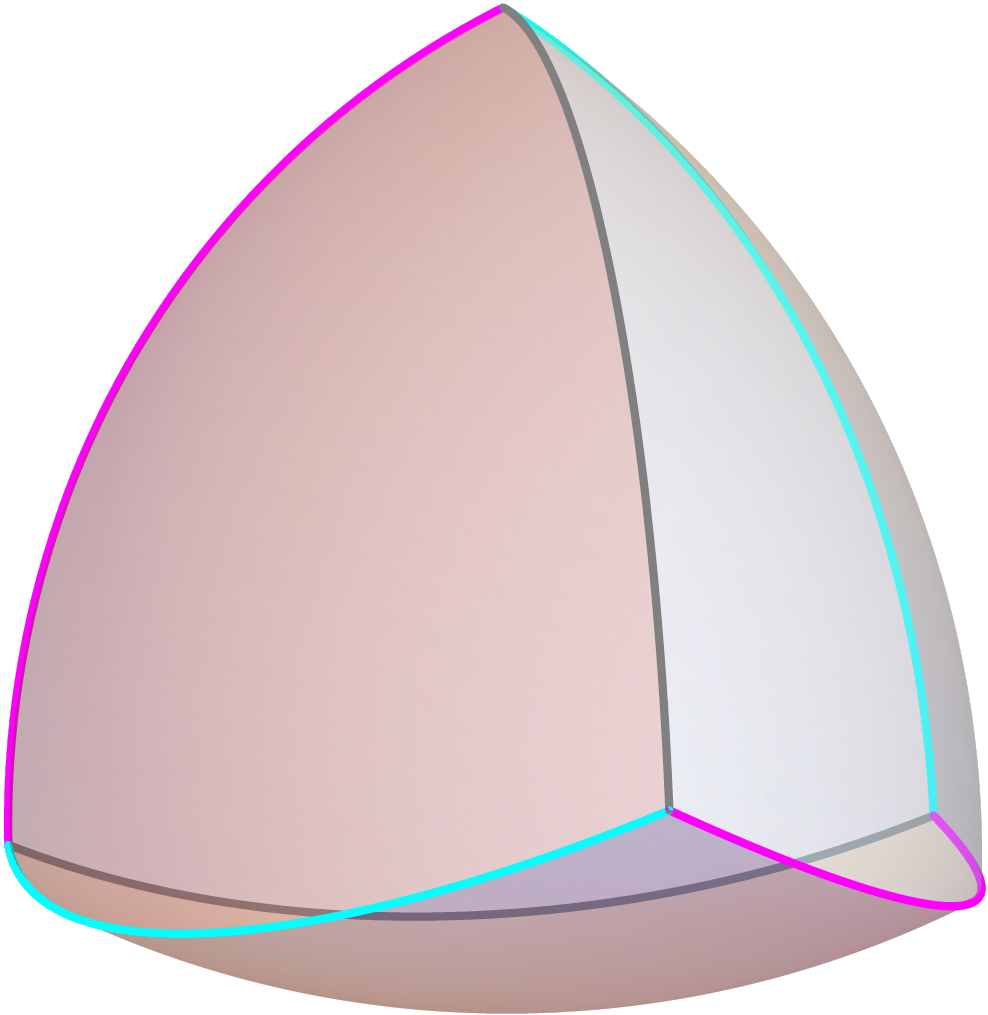}
 \hspace{.4in}
  \includegraphics[width=.45\textwidth]{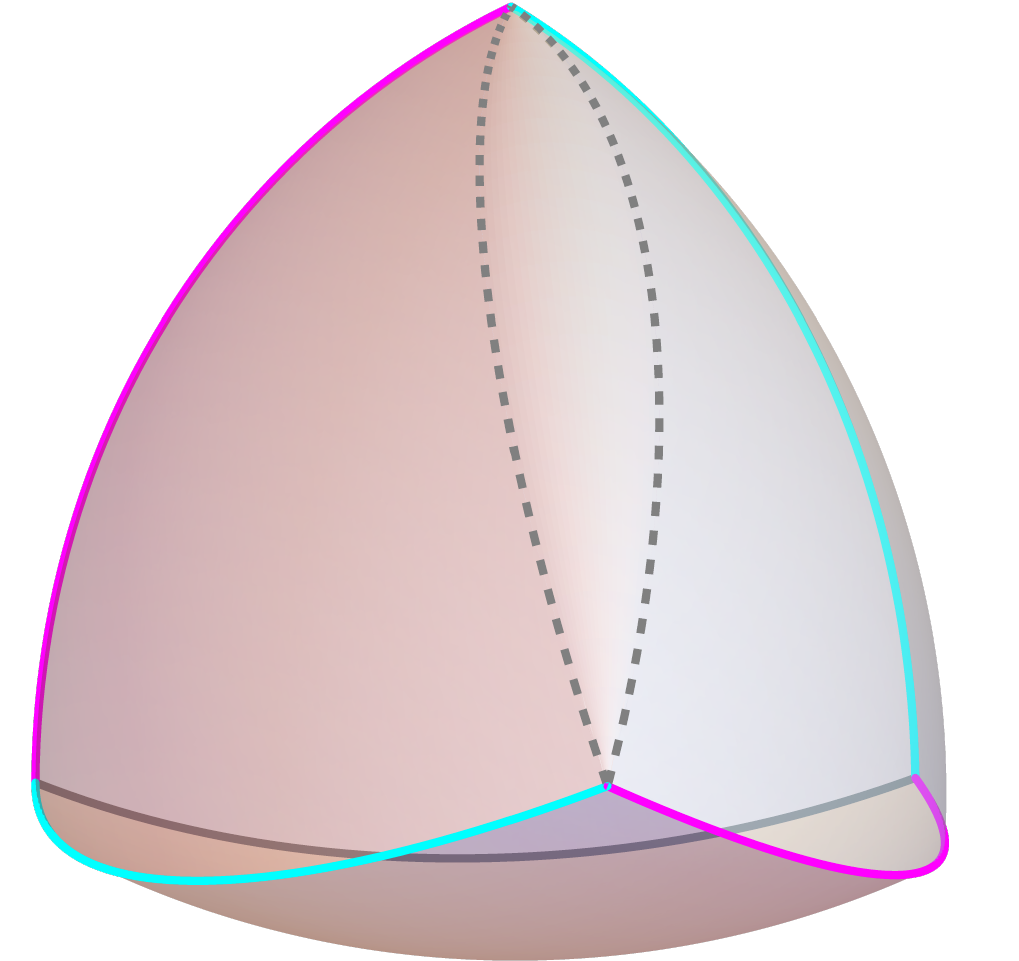}
 \caption{The figure on the left is a Reuleaux polyhedron $B(X)$, where $X$ is the set of vertices of a regular tetrahedron. Note that each dual edge pair of $B(X)$ is labeled with the same color. The figure on the right is the shape obtained from $B(X)$ by performing surgery on $\partial B(X)$ near one of its dual edges. The dashed curves are circular arcs of radius one which join two vertices of $B(X)$. The surgery procedure is to replace the region of $\partial B(X)$ bounded by these arcs with the surface obtained by rotating one of the arcs into the other about the line passing through the two vertices; this surface is a piece of a spindle torus, which is described in the appendix.}
 \label{ReulTetraFig}
\end{figure}
 \begin{figure}[h]
\centering
   \includegraphics[width=.45\textwidth]{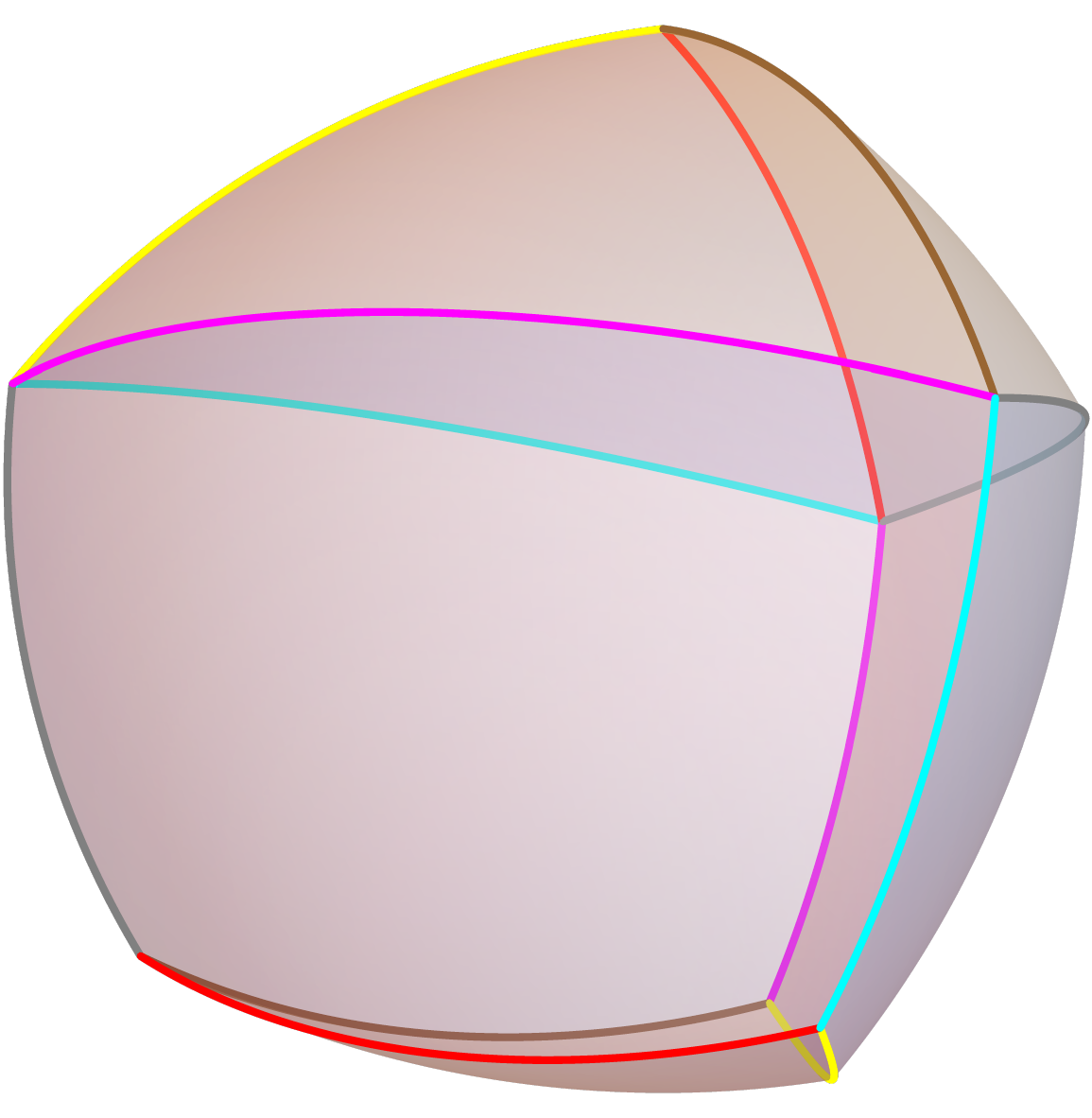}
 \hspace{.4in}
  \includegraphics[width=.45\textwidth]{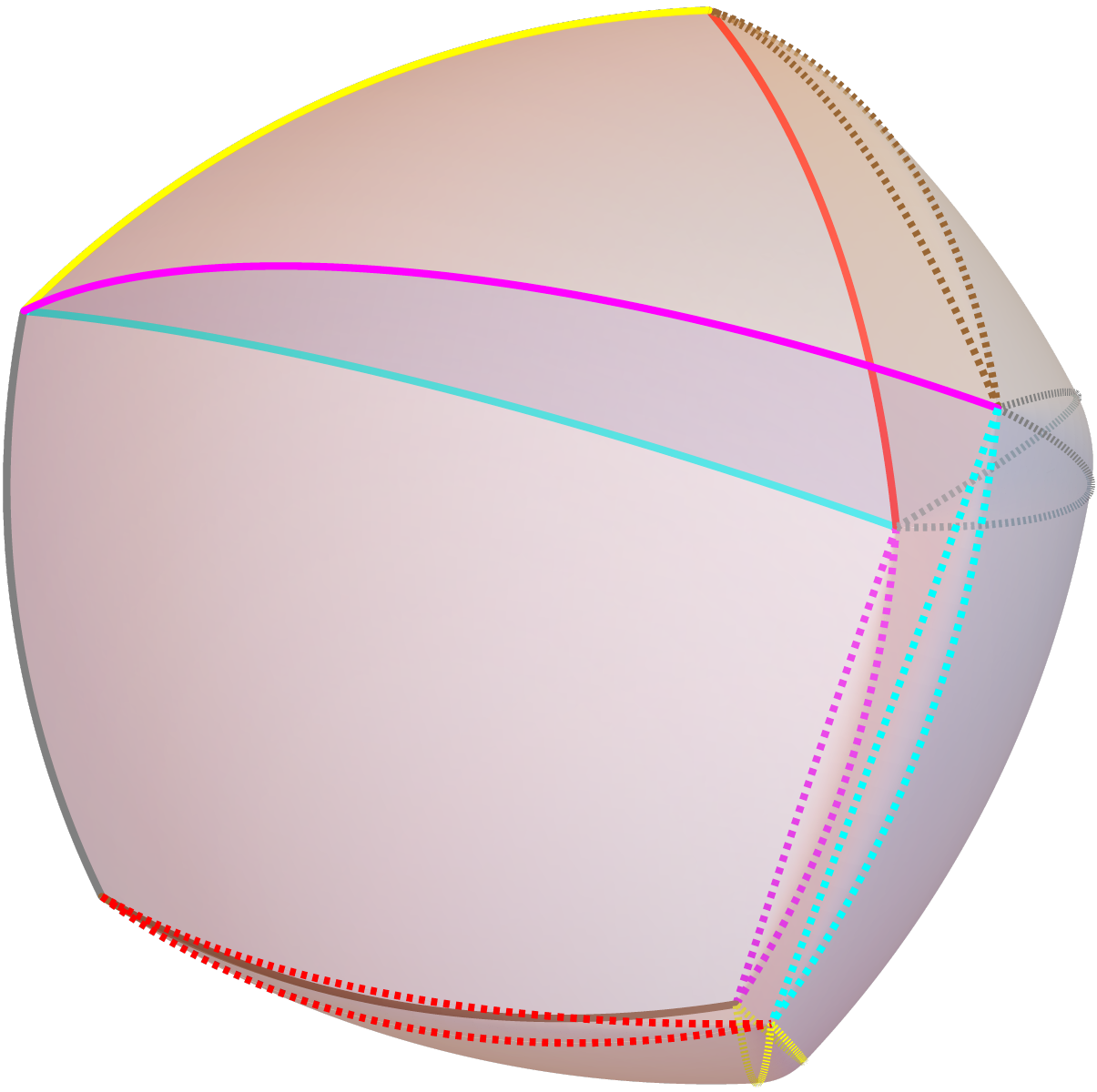}
 \caption{This is an example of a Reuleaux polyhedron $B(X)$ on the left and a corresponding Meissner polyhedron $M$ on the right. Again we have labeled the dual edge pairs of $B(X)$ with the same color.  Also note that 
 we have indicated the surgery bounds on the boundary of $M$ with dashed arcs. }
 \label{ElonTetraFig}
\end{figure}
\par A Reuleaux tetrahedron $B(X)$ does not have constant width.  However, Meissner and Shilling showed how to perform surgery on the boundary of $B(X)$ near three edges which  bound a common face or which meet in a common vertex to obtain two shapes which have constant width \cite{Meissner}.  Namely, they identified a portion of the boundary near a given edge and replaced it with a corresponding part of a spindle torus; see Figure \ref{ReulTetraFig} above and Figure \ref{SpindleFigure} in the appendix. The resulting figures are the two Meissner tetrahedra; see Figures \ref{MeissTetraO} and \ref{MeissTetraT} above.   Recently, Montejano and Rold\'an-Pensado figured out how to generalize Meissner and Shilling's construction to all Reuleaux polyhedra and obtain a new family of constant width shapes which we call Meissner polyhedra \cite{MR3620844}. 

\par Namely, they argued that if $B(X)$ is a Reuleaux polyhedron, a constant width body can be obtained by performing surgery near one edge in every dual edge pair. More specifically, suppose $X\subset \R^3$ has $m\ge 4$ points and has diameter one.  Then $B(X)$ has $m-1$ dual edge pairs 
$(e_1,e_1'),\dots, (e_{m-1},e_{m-1}')$. The {\it Meissner polyhedron} $M$ obtained by performing surgery on $B(X)$ near edges $e_1',\dots, e_{m-1}'$ is given by
$$
M=B\left(X\cup e_1\cup\cdots \cup e_{m-1}\right)
$$
(as described in section 4 of \cite{MR4775724}).  See Figure \ref{ElonTetraFig} for an example.  A detail we will use in our proof of Theorem \ref{thmA} below is that $X\cup e_1\cup\cdots \cup e_{m-1}\subset M$. 
 
 \par In addition, we will employ two basic facts about constant width shapes. First, a convex body $K$ has constant width if and only if $K=B(K)$.  Second, if $K_1$ and $K_2$ are constant width shapes with $K_1\subset K_2$, then $K_1=K_2$.  
 
\begin{proof}[Proof of Theorem \ref{thmA} ]  Suppose $X=\{x_j\}^m_{j=1}$ is an extremal subset of $ \R^3$.  Then $B(X)$ has dual edge pairs 
$(e_1,e_1'),\dots, e_{m-1},e'_{m-1})$ and $M=B(X\cup e_1\cup\dots\cup e_{m-1})$ is a Meissner polyhedron. We will argue that $M$ is extreme. To this end, we suppose 
\be\label{extremalequality}
M=(1-\lambda)M_0+\lambda M_1
\ee
for some $\lambda\in (0,1)$ and constant width $M_0, M_1\subset \R^3$. Recall that each $x_j\in M$.  In view of \eqref{extremalequality}, there are $x^0_j\in M_0$ and $x^1_j\in M_1$ with 
$$
x_j=(1-\lambda)x_j^0+\lambda x_j^1
$$
for each $j=1,\dots, m$. We set $X^0:=\{x^0_j\}^m_{j=1}\subset M_0$ and $X^1:=\{x^1_j\}^m_{j=1}\subset M_1$,
and we note that $\text{diam}(X^0)\le 1$ and $\text{diam}(X^1)\le 1$. 

\par An important observation is that if $|x_i-x_j|=1$, then 
\begin{align}
1&=|(1-\lambda)(x^0_i-x^0_j)+\lambda(x^1_i-x^1_j)|\\
&\le (1-\lambda)|x^0_i-x^0_j|+\lambda |x^1_i-x^1_j|\\
&\le (1-\lambda) \cdot 1+\lambda \cdot 1\\
&=1.
\end{align}
It follows that $|x^0_i-x^0_j|=|x^1_i-x^1_j|$; and since the closed unit ball in $\R^3$ is strictly convex, 
\be\label{xeyexjayidentity}
x_i-x_j=x^0_i-x^0_j=x^1_i-x^1_j.
\ee
Furthermore, $X^0$ and $X^1$ are extremal.  

\par Observe that  \eqref{xeyexjayidentity} is equivalent to 
\be\label{xeyexjayidentity2}
\begin{cases}
x_i-x^0_i=x_j-x^0_j\\
x_i-x^1_i=x_j-x^1_j
\end{cases}
\ee
whenever $|x_i-x_j|=1$. We claim that this identity actually holds for all $x_i$ and $x_j$. Since the graph consisting of the vertices $X$ and the edges of $B(X)$ is connected (Theorem 6.1 of \cite{MR2593321}), it suffices to show that \eqref{xeyexjayidentity2} holds for any $x_i$ and $x_j$ connected by an edge of $B(X)$. 

\par Note that if $x_i$ and $x_j$ are connected by an edge $e\subset \partial B(x_k)\cap \partial B(x_\ell)\cap B(X)$, then $|x_i-x_k|=|x_j-x_k|=1$  and \eqref{xeyexjayidentity2} holds for $x_i$ and $x_k$ and for $x_k$ and $x_j$. As a result, 
\be
\begin{cases}
x_i-x^0_i=x_k-x^0_k=x_j-x^0_j\\
x_i-x^1_i=x_k-x^1_k=x_j-x^1_j.
\end{cases}
\ee
We conclude that \eqref{xeyexjayidentity2} holds for all $x_i$ and $x_j$. 

\par  In addition, $B(X^0)$ has the same edges as $B(X)$ up to a translation of $c^0$.  That is, the edge of 
$B(X^0)$ are $(e_1-c^0,e'_1-c^0),\dots, (e_{m-1}-c^0,e'_{m-1}-c^0)$. We claim that
\be\label{edgeclaim}
e_j-c^0\subset M_0\text{ for $j=1,\dots,m-1$. }
\ee
To this end, we let $\gamma: [0,T]\rightarrow M$ be a parametrization of $e_j\subset \partial B(x)\cap \partial B(y)\cap B(X)$ for some $j$.  By \eqref{extremalequality}, there are $\gamma^0,\gamma^1: [0,T]\rightarrow M_0,M_1$ that satisfy
$$
\gamma(t)=(1-\lambda)\gamma^0(t)+\lambda \gamma^1(t)
$$
for $t\in [0,T]$. As $x=(1-\lambda)x^0+\lambda x^1$ for some $x^0\in X^0$ and $x^1\in X^1$, 
$$
\gamma(t)-x=(1-\lambda)(\gamma^0(t)-x^0)+\lambda (\gamma^1(t)-x^1).
$$
And since $|\gamma(t)-x|=1$,
$$
\gamma(t)-x=\gamma^0(t)-x^0.
$$

\par As $x=x^0+c^0$, $\gamma(t)-c^0=\gamma^0(t)\in M_0$ for all $t\in [0,T]$. Thus, $e_j-c^0\subset M_0$. Since $j$ was arbitrary,
the claim \eqref{edgeclaim} holds and   
$$
M_0\supset  X\cup e_1\cup\dots\cup e_{m-1}-c^0.
$$
It follows that 
$$
M_0=B(M_0)\subset B(X\cup e_1\cup\dots\cup e_{m-1}-c^0)=M-c^0.
$$
Likewise we conclude $M_1\subset M-c^1$. Since $M_0$ and $M_1$ have constant width, 
$$
M=M_0+c^0=M_1+c^1. 
$$
That is,  $M$ is extreme. 
\end{proof}
\begin{rem}
A similar argument can also be used to show that each Reuleaux polygon in $\R^2$ is extreme.
\end{rem}

\section{Rotated Reuleaux polygons}
In this section, we will use $y=(y_1,y_2,y_3)$ as coordinates for $\R^3$. Suppose that $C$ is a Reuleaux polygon in the $y_1y_3$-plane and denote $X$ as the set of vertices of $C$. We recall that  
\be\label{Kintersectionformula}
C=\bigcap_{x\in  X}D(x),
\ee
where $D(x)$ is the closed disk of radius one centered at $x$ in the $y_1y_3$-plane.  Further assume that $C$ is symmetric with respect to the $y_3$-axis and consider 
$$
K=\left\{y\in\R^3: \left(\sqrt{y_1^2+y_2^2},y_3\right)\in C\right\},
$$
which is the convex body obtained by rotating $C$ about the $y_3$-axis. As we will explain at the end of the following section, $K$ has constant width.  

\par For $x=(x_1,x_3)\in X$, define the circle 
\be\label{Circlecx}
c_x=\left\{y\in \R^3: y_3=x_3, \sqrt{y_1^2+y_2^2}=|x_1|\right\}
\ee
in $\R^3$.  The intersection of the closed balls of radius one centered at points on $c_x$ is given by
\be\label{BcxSpindleFormula}
B(c_x)=\left\{y\in \R^3: \left(\sqrt{y_1^2+y_2^2}+|x_1|\right)^2+(y_3- x_3)^2\le 1\right\}
\ee
(Corollary 2.7 of \cite{MR4775724}); $\partial B(c_x)$ is the inner portion of a spindle torus as described in the appendix. This intersection formula, the symmetry of $C$, and \eqref{Kintersectionformula} together imply 
\be\label{KeyMrepresentation}
K=\displaystyle\bigcap_{x\in X}B(c_x)=B\left(\bigcup_{x\in X}c_x\right).
\ee
This representation of $K$ will be crucial in proving that $K$ is extreme. 

\begin{proof}[Proof of Theorem \ref{thmB}]  1. As $C$ has an odd number of vertices, we may write $X=\{x_j\}^{2n}_{j=0}$. We can also choose $x_0$ as the lone vertex on the $y_3$-axis and select $x_j$ is adjacent to $x_{j+1}$ for $j=0,\dots, 2n$ where $x_{2n+1}=x_0$.  With this choice, 
\be\label{xjreflectionx2nplusoneminusj}
\text{$x_j$ is the reflection of $x_{2n+1-j}$ about the $y_3$-axis for each $j=1,\dots,n$.}
\ee
 We may also assume without any loss of generality that 
\be\label{FirstHalfVertPositiveX}
(x_j)_1\ge 0\quad\text{and}\quad (x_{2n+1-j})_1\le 0
\ee
 for $j=1,\dots, n$.  Our choices also lead to 
\be\label{ReulEdgeCond}
|x_j-x_{n+j}|=|x_j-x_{n+j+1}|=1
\ee 
for $j=1,\dots,n$. See Figure \ref{LabeledAxialFig} for an example. 
\begin{figure}[h]
\centering
 \includegraphics[width=.45\textwidth]{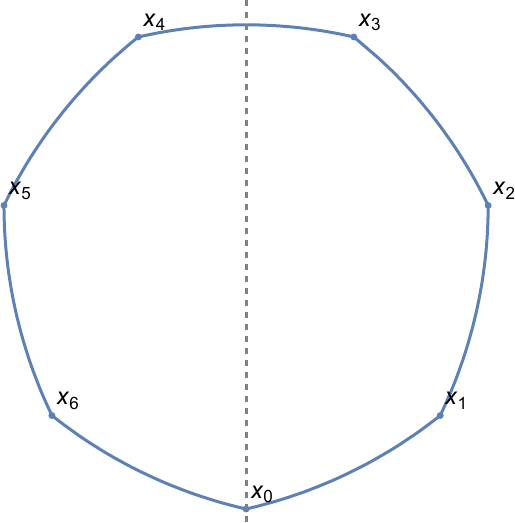}
 \hspace{.3in}
  \includegraphics[width=.45\textwidth]{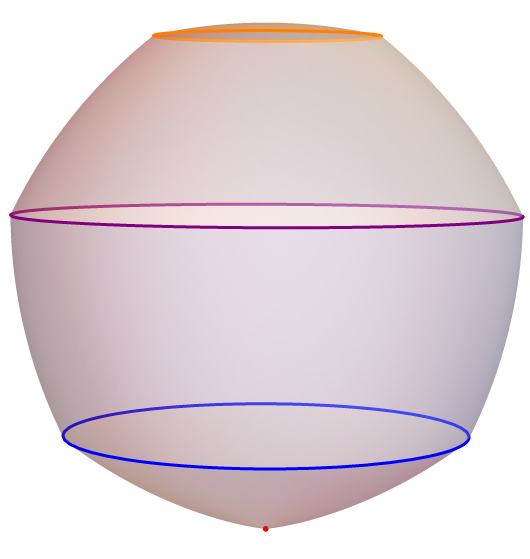}
 \caption{This figure illustrates the notation we have used in our proof of Theorem \ref{thmB}. On the left, a symmetric Reuleaux heptagon with vertices $\{x_0,\dots,x_6\}$ is shown. Note that $x_0$ is on the vertical axis and conditions \eqref{xjreflectionx2nplusoneminusj}, \eqref{FirstHalfVertPositiveX}, and \eqref{ReulEdgeCond} are satisfied. On the right is the corresponding surface of revolution which bounds a constant width figure in $\R^3$. The circles $c_{x_1}, c_{x_6}$ are labeled blue, $c_{x_2}, c_{x_5}$ are labeled purple, and $c_{x_3}, c_{x_4}$ are labeled orange; the circle $c_{x_0}$ is a point labeled red.}\label{LabeledAxialFig}
\end{figure}

\par 2. Assume $K$ is given by \eqref{KeyMrepresentation}, $\lambda\in (0,1)$, and $K_0,K_1\subset \R^3$ are constant width shapes for which 
\be\label{mlamidentity}
K=(1-\lambda)K_0+\lambda K_1.
\ee
There are $X_0=\{x^0_j\}^{2n}_{j=0}\subset K_0$, $X_1=\{x^1_j\}^{2n}_{j=1=0}\subset K_1$, so that 
$$
x_j=(1-\lambda)x_j^0+\lambda x^1_j
$$
for $j=0,\dots, 2n$. The diameters of $X_0$ and $X_1$ are both at most one; and since $X$ has diameter one, the diameters of $X_0$ and $X_1$ are equal to one. As in the previous proof,
\be\label{2ndxeyexjayidentity}
x_i-x_j=x^0_i-x^0_j=x^1_i-x^1_j
\ee
whenever $|x_i-x_j|=1$.  

\par In order to see that the above relation holds for all $i$ and $j$, we just need to show it does for all $i$ and $j=i+1$. For a given $i$, there is $k$ with $|x_i-x_{k}|=|x_i-x_{k+1}|=1.$  Note that $|x_{k+1}-x_{i+1}|=1$, as well.  Therefore, 
$$
x_i-x_{i+1}=(x_{i}-x_{k+1})+(x_{k+1}-x_{i+1})=(x_{i}^0-x_{k+1}^0)+(x_{k+1}^0-x_{i+1}^0)=x_i^0-x_{i+1}^0.
$$
Likewise, $x_i-x_{i+1}=x_i^1-x_{i+1}^1$.  We conclude that \eqref{2ndxeyexjayidentity} holds for all $i,j$. 

\par  It follows that $x_i-x^0_i$ and $x_i-x^0_i$ are independent of $i$, so there are $c^0,c^1\in \R^3$ with 
$$
c^0=x^0_i-x_i\quad \text{and}\quad c^1=x^1_i-x_i
$$
for all $i=0,\dots, 2n$. And as $x_0=(1-\lambda)x_0^0+\lambda x_0^1=x_0+(1-\lambda)c^0+\lambda c^1$, 
\be\label{czeroconeidentity}
(1-\lambda)c^0+\lambda c^1=0.
\ee

\par 3.  Next we claim that 
\be\label{circlexjayclaim}
\text{$c_{x_j}+c^0\subset K_0$ and $c_{x_j}+c^1\subset K_1$ for each $j=0,\dots, n$}.
\ee
Once we have verified this, then 
$$
\bigcup^m_{i=1}(c_{x_i}+c^0)=\bigcup^m_{i=1}c_{x_i}+c^0\subset K_0.
$$
Further, 
$$
K_0=B(K_0)\subset B\left(\bigcup^m_{i=1}c_{x_i}+c^0\right)  =B\left(\bigcup^m_{i=1}c_{x_i}\right)  +c^0=K+c^0.
$$
In the last equality, we used \eqref{KeyMrepresentation}. And as $K+c^0$ has constant width,  $K_0=K+c^0$. 
The same argument would give $K_1=K+c^1$, and in turn $K_1=K_0+c^1-c^0$.  Therefore, it suffices to verify 
the claim \eqref{circlexjayclaim}.

\par 4. We will explain how to establish \eqref{circlexjayclaim} by induction along the finite sequence
\be\label{alternatingJsequence}
j=0, n, 1, n-1, 2, n-2, ...
\ee
To this end, we will employ the parametrizations 
$$
\gamma_j(t)=(r_j\cos(t),r_j\sin(t),h_j), \quad t\in [0,2\pi]
$$
of the circles $c_{x_j}$ for $j=1,\dots, n$. Here $r_j=(x_j)_1$ is the radius and $h_j=(x_j)_3$ is the height of $c_{x_j}$ as measured by the $y_3$-axis. 
Note in particular that
\be\label{gammaj0andothergammapi}
\gamma_j(0)=x_j\quad\text{and}\quad \gamma_j(\pi)=x_{2n+1-j}
\ee
for $j=1,\dots, n$ by \eqref{xjreflectionx2nplusoneminusj}. 

\par For $j=0$,  $\gamma_j(t)=x_0$ for all $t\in [0,2\pi]$ since $(x_0)_1=0$. The claim  \eqref{FirstHalfVertPositiveX} holds in this case since 
$c_{x_0}=\{x_0\}$, $x_0+c^0=x^0_0\in K_0$, and $x_0+c^1=x^1_0\in K_1$.  Let us now suppose the claim holds for some $j=k\in \{0,\dots, n\}$. 
We will show that this implies that it also does for $j=n-k$. By \eqref{mlamidentity}, there are mappings $\gamma^0_{n-k}: [0,2\pi]\rightarrow K_0$ and $\gamma^1_{n-k}: [0,2\pi]\rightarrow K_1$ with 
\be
\gamma_{n-k}(t)=(1-\lambda)\gamma^0_{n-k}(t)+\lambda \gamma^1_{n-k}(t)
\ee
for each $t\in [0,2\pi]$. In view of \eqref{ReulEdgeCond} and \eqref{gammaj0andothergammapi}, 
\be\label{gammanminuskEq}
|\gamma_{n-k}(t)-\gamma_k(t+\pi)|=|\gamma_{n-k}(0)-\gamma_k(\pi)|=|x_{n-k}-x_{2n+1-k}|=1
\ee
 for all $t\in [0,2\pi]$. 
 
\par Our induction hypothesis is that $\gamma_k(t)+c^0\in K_0$ and $\gamma_k(t)+c^1\in K_1$ for all $t\in [0,2\pi]$. 
 Therefore,  
\be\label{gammanminuskINEq}
 |\gamma^0_{n-k}(t)-(\gamma_k(t+\pi)+c^0)|\le 1\quad \text{and}\quad |\gamma^1_{n-k}(t)-(\gamma_k(t+\pi)+c^1)|\le 1
\ee
 for all $t\in [0,2\pi]$.
 Also note that  \eqref{czeroconeidentity} allows us to write 
\begin{align}
&\gamma_{n-k}(t)-\gamma_k(t+\pi)&\\
&\hspace{.5in}=(1-\lambda)\left(\gamma^0_{n-k}(t)-(\gamma_k(t+\pi)+c^0)\right)+\lambda \left(\gamma^1_{n-k}(t)-(\gamma_k(t+\pi)+c^1)\right).
\end{align}
By \eqref{gammanminuskEq} and \eqref{gammanminuskINEq}, 
  $$
  \gamma_{n-k}(t)-\gamma_k(t+\pi)=\gamma^0_{n-k}(t)-(\gamma_k(t+\pi)+c^0)=\gamma^1_{n-k}(t)-(\gamma_k(t+\pi)+c^1).
  $$
  That is, 
  $$
c^0 + \gamma_{n-k}(t)= \gamma^0_{n-k}(t)\in K_0\quad \text{and}\quad  c^1 + \gamma_{n-k}(t)= \gamma^1_{n-k}(t)\in K_1
  $$
  for all $t\in [0,2\pi]$. In particular, this proves the claim for $j=n-k$. The proof that the claim holding for $j=n-k$ implies that it does for $j=k+1$ follows 
  similarly, so we conclude by induction. 
  \end{proof}

\section{Support function}\label{SupportSection}
In the previous sections, we used intersection properties of constant width shapes to study extreme shapes.  For the rest of this note, 
we will employ the support function. This section is a brief overview which is relevant for our purposes.  We will state the majority of results for constant width shapes in $\R^n$, but we only have $n=2$ and $n=3$ in mind. Good general references for the support function are \cite{MR1216521,MR920366}. And some of the computations below for the support function of constant width bodies were done by Howard \cite{MR2233133}. A reader familiar with the support function of a constant width body
may skip ahead to the following sections. 

\par We define the {\it support function} of a convex body $K\subset \R^n$ as 
$$
H(u)=\max_{x\in K}x\cdot u, \quad u\in \R^n.
$$
Note that $H$ is continuous, convex, and positively homogeneous. We can interpret $H$ geometrically as follows: for $x\in K$ and $u\in \Snmo$, $H(u)-x\cdot u$ is the 
distance from $x$ to the supporting plane of $M$ with outward normal $u$.  The formula   
$$
K=\bigcap_{u\in \Snmo}\{x\in \R^n: x\cdot u\le H(u)\}
$$
also allows us to recover $K$ from $H$. 

\par  If we write $H_K$ for the support function associated with a convex body $K$, 
then 
$$
H_{K_1+K_2}=H_{K_1}+H_{K_2}.
$$
It turns out that $K$ has constant width if and only if $K+(-K)=B(0)$. Therefore,  $K$ has constant width if and only if 
its support function $H$ satisfies
\be\label{Hconstantwidthcondition}
H(u)+H(-u)=|u|, \quad u\in \R^n. 
\ee
Furthermore, $H_{(1-\lambda)K_1+\lambda K_2}=(1-\lambda)H_{K_1}+\lambda H_{K_2}$ for $\lambda\in [0,1]$.  It follows that $(1-\lambda)K_1+\lambda K_2$ has constant width whenever $K_1$ and $K_2$ do, as noted in the introduction. 
\par We will specialize to constant width bodies for the remainder of this section. 
\\\\
\noindent {\bf Differentiability}. Rademacher's theorem implies that the Hessian $D^2H(u)$ exists and has nonnegative eigenvalues for a.e. $u\in \R^n$. In view of \eqref{Hconstantwidthcondition}, 
\be\label{HessianComputationZero}
D^2H(u)+D^2H(-u)=\frac{1}{|u|}\left(I-\frac{u\otimes u}{|u|^2}\right)
\ee
for a.e. $u\neq 0$. Here $I$ is the $n\times n$ identify matrix.  Note in particular that $D^2H$ is essentially bounded from above away from the origin. As a result, $H$ is continuously differentiable away from the origin and $DH$ is Lipschitz continuous on $|u|\ge r$ for any $r>0$. 
\\\\
\noindent {\bf The inverse Gauss map}.  Since $H$ is positively homogeneous
\be\label{HomogeneousOneH}
H(u)=DH(u)\cdot u
\ee
for all $u\in \R^n\setminus\{0\}$. Moreover, $DH: \Snmo \to \partial K$
is surjective. And as $K$ is strictly convex, $DH(u)$ is the unique $x\in \partial K$ for which $H(u)=x\cdot u$.  Therefore, $DH(u)$ is the point on $\partial K$ which has outward  unit normal $u\in \Snmo$; so if $\partial K$ is smooth, $DH$ is the inverse of the Gauss map. 
 \\\\
\noindent {\bf The restriction of $H$ to $\Snmo$}. We recall the tangent space of $\Snmo$ at $u$ is simply $u^\perp=\{v\in \R^n: u\cdot v=0\}$. 
Now consider $H$ as a function on $\Snmo$
$$
h=H|_{\Snmo}.
$$
As mentioned above, $h: \Snmo\rightarrow \R$ is continuously differentiable and $\nabla h:  \Snmo\rightarrow \R^n$ is Lipschitz continuous.  A direct computation gives 
$$
DH(u)=\nabla h(u)+h(u)u
$$
for each $u\in \Snmo$.   Moreover, 
\be\label{HessianComputation}
\nabla(DH)(u)=\nabla^2 h(u)+h(u)\id_{u^\perp}
\ee
exists and is a symmetric, nonnegative definite operator on $u^\perp$ for a.e. $u\in \Snmo$. 
\begin{rem}\label{RecoveryRemark}
Conversely, if $h: \Snmo\rightarrow \R$ is continuously differentiable, $h(u)+h(-u)=1$ for all $u\in\Snmo$, and 
$\nabla^2h(u)+h(u)\id_{u^\perp}$ is nonnegative definite for a.e. $u\in\Snmo$, then  
$$
H(u)=
\begin{cases}
|u|h\left(u/|u|\right), \quad & u\neq 0\\
0, \quad & u=0
\end{cases}
$$
is the support function of a constant width body. 
\end{rem}  
\noindent {\bf Principal radii of curvature bounds}. If $K$ is smooth and
$\nabla^2h(u)+h(u)\id_{u^\perp}$
is positive definite for each $u\in\Snmo$, this mapping is the differential of the inverse Gauss map $DH$ at $u$.  
Its inverse is the operator associated with the second fundamental form of $\partial K$ at $DH(u)$.   As a result, the eigenvalues of $\nabla^2h(u)+h(u)\id_{u^\perp}$ are the principal radii of curvature of  $\partial K$ at $DH(u)$. 

\par As noted above, for a general $K$, $\nabla^2h(u)+h(u)\id_{u^\perp}$ exists and is symmetric and  nonnegative definite for a.e. $u\in \Snmo$. As a result, we will also call its eigenvalues principal radii of curvature. Since $h(u)+h(-u)=1$ for all $u\in \Snmo$, 
\be\label{RplusRminusOne}
\left(\nabla^2h(u)+h(u)\id_{u^\perp}\right)+\left(\nabla^2h(-u)+h(-u)\id_{u^\perp}\right)=\id_{u^\perp}
\ee
 a.e. $u\in \Snmo$. And as $\nabla^2h(-u)+h(-u)\id_{u^\perp}$ is essentially nonnegative definite, 
$$
\text{the principal radii of curvature of $K$ belong to the interval $[0,1]$ }
$$
for a.e. $u\in \Snmo$.
\\\\ 
\noindent {\bf Polar coordinates}.
 Suppose $n=2$ and set $u(\theta)=(\cos(\theta),\sin(\theta))$ for $\theta \in [0,2\pi]$.  Note $\theta$ is the standard polar coordinate on $\mathbb{S}^1$.  
We'll write $h(\theta)$ for $H(u(\theta))$. Observe
$$
h(\theta)+h(\theta+\pi)=1
$$ 
and 
\begin{align}
DH(u(\theta))&=h(\theta)u(\theta)+h'(\theta)u'(\theta)
\end{align}
is a parametrization of $\partial K$ in polar coordinates. As,  
$$
\frac{d}{d\theta}DH(u(\theta))=(h''(\theta)+h(\theta))u'(\theta)
$$
for a.e. $\theta \in [0,2\pi]$, 
$$
0\le h''(\theta)+h(\theta)\le 1$$
for a.e. $\theta \in [0,2\pi]$. 
\\\\
\noindent {\bf Spherical coordinates}. Suppose $n=3$ and recall that $\Stwo$ may be parametrized with  
$$
u(\theta,\phi)=(\sin\phi\cos\theta ,\sin\phi\sin\theta, \cos\phi)
$$
for $\theta\in [0,2\pi]$ and $\phi\in [0,\pi]$.  These are standard spherical coordinates on $\Stwo$, where the $x_3$-axis is the polar axis, $\phi$ is the polar angle, and $\theta$ is the azimuthal angle. Moreover, 
$\{u_\theta,u_\phi\}$ is an orthogonal basis for $u^\perp$. 

\par For ease of notation, we will write $h(\theta,\phi)$ for $H(u(\theta,\phi))$. Note that 
$$
h(\pi+\theta,\pi-\phi)+h(\theta,\phi)=1
$$
and 
\be\label{SphericalCoordParam}
DH(u)=\frac{h_\theta}{(\sin\phi)^2} u_\theta+h_\phi u_\phi+h u
\ee
parametrizes $\partial K$.  Direct computation also gives 
$$
(DH(u))_\theta=\left(\frac{h_{\theta\theta}}{(\sin\phi)^2}+\frac{h_\phi}{\tan\phi} +h\right)u_\theta+\left(\sin\phi\partial_{\phi}\left(\frac{h_\theta}{\sin\phi}\right)\right)u_\phi
$$
and 
$$
(DH(u))_\phi=\left(\frac{1}{\sin\phi}\partial_{\phi}\left(\frac{h_\theta}{\sin\phi}\right)\right)u_\theta+(h_{\phi\phi}+h)u_\phi.
$$
\par As a result, 
\be\label{SphericalCoordIneq}
\left(
\begin{array}{cc}
\displaystyle\frac{h_{\theta\theta}}{(\sin\phi)^2}+\frac{h_\phi}{\tan\phi} +h &\displaystyle \sin\phi\partial_{\phi}\left(\frac{h_\theta}{\sin\phi}\right)  \\
\displaystyle\frac{1}{\sin\phi}\partial_{\phi}\left(\frac{h_\theta}{\sin\phi}\right) & h_{\phi\phi}+h
\end{array}
\right)
\ee
is the matrix representation of $\nabla(DH)(u)=\nabla^2 h(u)+h(u)\id_{u^\perp}$ in the basis $\{u_\theta,u_\phi\}$. 
It follows that the above matrix is diagonalizable and its eigenvalues belong to $[0,1]$ for a.e. $\theta\in [0,2\pi]$ and $\phi\in [0,\pi]$.
\\\\
\noindent {\bf Axial symmetry}. Suppose $C$ is a constant width shape in the $x_1x_3$-plane and that $C$ is symmetric about the $x_3$-axis. Then its support function satisfies 
\be\label{HKsymmetric2D}
H_C(u_1,u_3)=H_C(-u_1,u_3)
\ee
for $(u_1,u_3)\in \R^2$. Consider the convex body 
$$
K=\left\{x\in\R^3: \left(\sqrt{x_1^2+x_2^2},x_3\right)\in C\right\}
$$
obtained by rotating $C$ about the $x_3$-axis. It is routine to check that 
\be\label{HMsymmetric3D}
H_K(v)=H_C\left(\sqrt{v_1^2+v_2^2},v_3\right)
\ee
for $v\in \R^3$. It follows easily from \eqref{HKsymmetric2D} and \eqref{HMsymmetric3D} that $H_K$ satisfies \eqref{Hconstantwidthcondition}; consequently, $K$ has constant width.

\par In spherical coordinates, 
\be\label{AxialhmhkEq}
h_K(\theta,\phi)=H_K(u(\theta,\phi))=H_C(\sin\phi,\cos\phi)=h_C(\pi/2-\phi)
\ee
is independent of $\theta$. Therefore, $h(\phi)=h_K(\theta,\phi)$ satisfies 
$$
h(\pi-\phi)+h(\phi)=1
$$
for all $\phi\in [0,\pi]$, and the matrix representation \eqref{SphericalCoordIneq} is 
$$
\left(
\begin{array}{cc}
\displaystyle\frac{h'}{\tan\phi} +h &\displaystyle 0  \\
\displaystyle0 & h''+h
\end{array}
\right)
$$
for a.e. $\phi\in [0,\pi]$.

\section{Axially symmetric shapes}
Kallay showed that a constant width $C\subset \R^2$ is extreme if and only if its support function $h$ in polar coordinates satisfies
$$
 h''(\theta)+h(\theta)\in\{0,1\}\quad\text{for a.e. $\theta\in [0,2\pi]$}.
 $$
 (Theorem 5 of \cite{MR350618}). We will use this result in our proof of Theorem \ref{thmC}. Let us first make a basic observation. 
\begin{lem}
Suppose that $C\subset \R^2$ has constant width and is symmetric with respect to the $x_2$-axis.  If $C$ is not extreme, then  
$C=(1-\lambda)C_++\lambda C_-$ for some $\lambda\in (0,1)$, where $C_\pm\subset \R^2$ are constant width, symmetric with respect to the $x_2$-axis, and 
$C_+$ is not a translate of $C_-$. 
\end{lem}
\begin{proof}
By assumption,  $C=(1-\lambda)C_0+\lambda C_1$ for some $\lambda\in (0,1)$ and constant width $C_0$, $C_1$ which are not translates of each other. Let $T$ denote reflection about the $x_2$-axis. Since $C$ is invariant under $T$, $C=(1-\lambda)TC_0+\lambda TC_1.$  Therefore, 
$$
C=(1-\lambda)C_++\lambda C_-,
$$
where 
$$
C_+=\frac{C_0+TC_0}{2}\quad \text{and}\quad C_-=\frac{C_0+TC_0}{2}.
$$
\par Note $C_\pm$ are both symmetric and have constant width. If $C_+=C_-+b$, then for all $x_0,y_0\in C_0$ and $x_1,y_1\in C_1$, 
$$
x_0+Ty_0=x_1+Ty_1+2b.
$$
That is, $x_0=x_1+(2b +Ty_1-Ty_0)$ for all $x_0\in C_0$ and $x_1\in C_1$. This would imply that $C_0$ and $C_1$ are translates. 
\end{proof}
\begin{proof}[Proof of Theorem \ref{thmC}]
We will suppose $K\subset\R^3$ has constant width, is axially symmetric about the 
$x_3$-axis, and has generating shape $C$ in the $x_1x_3$-plane which is symmetric about the $x_3$-axis.  Recall that we aim to show that $K$ is extreme if and only if $C$ is extreme. \\
\par $(\Longrightarrow)$ Suppose $C$ is not extreme. The above lemma implies that $C=(1-\lambda) C_++\lambda C_-$ for some $\lambda\in (0,1)$ and symmetric constant width $C_\pm$ in the $x_1x_3$-plane which are not translates. As, $H_C=(1-\lambda)H_{C_+}+\lambda H_{C_-}$,  
\begin{align}
H_K(v)&=H_C\left(\sqrt{v_1^2+v_2^2},v_3\right)\\
&=(1-\lambda)H_{C_+}\left(\sqrt{v_1^2+v_2^2},v_3\right)+\lambda H_{C_-}\left(\sqrt{v_1^2+v_2^2},v_3\right)\\
&=(1-\lambda)H_{K_+}(v)+ \lambda H_{K_-}(v).
\end{align}
Here $K_\pm$ is the rotation of $C_{\pm}$ about the $x_3$-axis. Thus, $K=(1-\lambda)K_++\lambda K_-$. If $K_+=K_-+b$ for some $b\in \R^3$, then $C_+=C_-+(b_1,b_3)$. Thus $K_+$ is not a translate of $K_-$, so $K$ is not extreme. 
It follows that if $K$ is extreme, then $C$ is extreme. 
\\
\par $(\Longleftarrow)$ Suppose $C$ is extreme. Further assume that 
there are constant width $K_0,K_1\subset \R^3$ and $\lambda\in (0,1)$  with 
\be\label{ExtEqnKaxial}
K=(1-\lambda)K_0+\lambda K_1.
\ee
Let $H$, $H_0$, and $H_1$ be the support functions of $K, K_0$, and $K_1$, respectively. Then 
$$
H(u)=(1-\lambda)H_0(u)+\lambda H_1(u)
$$
for all $u\in \R^3$. Differentiating both sides of this equation at $u=e_3$ gives  
$$
DH(e_3)=(1-\lambda)DH_0(e_3)+\lambda DH_1(e_3).
$$

\par In view of \eqref{ExtEqnKaxial}, 
\be\label{ExtEqnKaxial2}
K-DH(e_3)=(1-\lambda)(K_0-DH_0(e_3))+\lambda (K_1-DH_1(e_3)).
\ee
Note that the origin is the point of the boundary of $K-DH(e_3)$, $K_0-DH_0(e_3)$, and $K_0-DH_0(e_3)$ 
which has outward unit normal $e_3$. Since $K$ is axially symmetric, $DH(e_3)$ belongs to the $x_3$-axis; therefore, $ K-DH(e_3)$ is still axially symmetric.  As $K_0-DH_0(e_3)$ and $K_1-DH_1(e_3)$ are translates if and only if $K_0$ and $K_1$ are translates, we will assume going forward that the origin is the point of the boundaries of $K, K_0$, and $K_1$ which has outward unit normal $e_3$.  With this assumption, 
\be\label{AxialNormalization}
DH(e_3)=DH_0(e_3)=DH_1(e_3)=0.
\ee

\par We will represent the support functions of $K, K_0$, and $K_1$ in spherical coordinates $(\theta,\phi)$ as in the previous section by setting $h(\theta, \phi)=H(u(\theta,\phi))$ and $h^i(\theta, \phi)=H_i(u(\theta,\phi))$ for $i=0,1$. We first note that since $K$ is axially symmetric, $h(\theta, \phi)=h(\phi)$ is independent of $\theta$. Next we have 
$$
h(\phi)=(1-\lambda)h^0(\theta,\phi)+\lambda h^1(\theta,\phi).
$$
Moreover, \eqref{AxialNormalization} implies
\be\label{AxialNormalization2}
h'(0)=DH(e_3)\cdot u_\phi(\theta,0)=0\quad\text{and}\quad 
h^i_\phi(\theta,0)=DH_i(e_3)\cdot u_\phi(\theta,0)=0
\ee
for all $\theta\in [0,2\pi]$ and $i=0,1$.

\par Recall that since $C$ is extreme, 
\be\label{KallayCond}
h''(\phi)+h(\phi)\in \{0,1\}
\ee
for almost all $\phi\in [0,\pi]$. This follows from \eqref{AxialhmhkEq} and Kallay's theorem. Set 
$$
R=\left(
\begin{array}{cc}
\displaystyle\frac{h'}{\tan\phi} +h &\displaystyle 0  \\
\displaystyle0 & h''+h
\end{array}
\right)
\;\;\text{and}\;\;
R^i=\left(
\begin{array}{cc}
\displaystyle\frac{h^i_{\theta\theta}}{(\sin\phi)^2}+\frac{h^i_\phi}{\tan\phi} +h^i &\displaystyle \sin\phi\partial_{\phi}\left(\frac{h^i_\theta}{\sin\phi}\right)  \\
\displaystyle\frac{1}{\sin\phi}\partial_{\phi}\left(\frac{h^i_\theta}{\sin\phi}\right) & h^i_{\phi\phi}+h
\end{array}
\right)
$$
for a.e. $\theta\in [0,2\pi]$ and $\phi\in[0,\pi]$.  Note 
$$
R=(1-\lambda)R^0+\lambda R^1,
$$  
and recall that the eigenvalues of these matrices 
belong to $[0,1]$. 

\par Observe 
\be\label{PrinCurvExtreme}
h''(\phi)+h(\phi)=(1-\lambda)(h^0_{\phi\phi}(\theta,\phi)+h^0(\theta,\phi))
+\lambda(h^1_{\phi\phi}(\theta,\phi)+h^1(\theta,\phi))
\ee
and 
\be\label{FirstRevalueBound}
0\le R^ie_2\cdot e_2= h^i_{\phi\phi}(\theta,\phi)+h^i(\theta,\phi)\le 1
\ee
for a.e. $\theta\in [0,2\pi]$ and $\phi\in[0,\pi]$.  In view of \eqref{KallayCond}, 
$$
h''(\phi)+h(\phi)=h^i_{\phi\phi}(\theta,\phi)+h^i(\theta,\phi)
$$
for a.e. $\theta\in [0,2\pi]$ and $\phi\in[0,\pi]$. It follows that there are continuously differentiable 
$a_i, b_i: [0,2\pi]\rightarrow \R$ with 
\be\label{FirstHeyeReduction}
h^i(\theta,\phi)=h(\phi)+a_i(\theta)\cos(\phi)+b_i(\theta)\sin(\phi)
\ee
for all $\theta\in [0,2\pi]$ and $\phi\in[0,\pi]$.

\par We also have 
$$
h''+h=\left(\frac{R+R^t}{2}\right)e_2\cdot e_2=(1-\lambda)\left(\frac{R^0+(R^0)^t}{2}\right)e_2\cdot e_2+\lambda \left(\frac{R^1+(R^1)^t}{2}\right)e_2\cdot e_2.
$$
Since the eigenvalues of the symmetric matrix $(R^i+(R^i)^t)/2$ belong to $[0,1]$ and $h''+h\in \{0,1\}$ a.e.,  $e_2$ is an eigenvector of $(R^i+(R^i)^t)/2$  with eigenvalue $h''+h$  for a.e. $\theta\in [0,2\pi]$ and $\phi\in[0,\pi]$. That is, 
$$
\left(\frac{R^i+(R^i)^t}{2}\right)e_2
=\left(\begin{array}{c}
\frac{1}{2}\left(\frac{1}{\displaystyle\sin\phi}+\sin\phi\right)\partial_{\phi}\left(\displaystyle\frac{h^i_\theta}{\sin\phi}\right)\\
h^i_{\phi\phi}+h^i
\end{array}\right)=(h''+h)e_2.
$$
As a result, 
$$
\partial_{\phi}\left(\frac{h^i_\theta}{\sin\phi}\right)=0
$$
for a.e. $\theta\in [0,2\pi]$ and $\phi\in[0,\pi]$.

\par In view of \eqref{FirstHeyeReduction}, 
\be
0=\partial_{\phi}\left(\frac{h^i_\theta}{\sin\phi}\right)=a_i'(\theta)\left(\frac{ 1}{\tan\phi}\right)'
\ee
for a.e. $\theta\in [0,2\pi]$  and $\phi\in[0,\pi]$. Thus, $a_i(\theta)$ is equal to some constant $c_i\in \R$. And by \eqref{AxialNormalization2}, 
$$
0=h^i_\phi(\theta,0)=h'(0)+b_i(\theta)=b_i(\theta)
$$
for all $\theta$. Therefore, 
\be\label{SecondHeyeReduction}
h^i(\theta,\phi)=h(\phi)+c_i\cos(\phi)
\ee
for $i=0,1.$ Consequently, $H_{i}(u)=H(u)+c_iu_3$ for $u\in \R^3$ and in turn 
$$
K_i=K+c_ie_3
$$
for $i=0,1.$ It follows that $K$ is extreme. 
\end{proof}

\section{Conjecture}
We will 
write $H$ for the support function of a constant width $K\subset \R^3$ and $h=H|_{\Stwo}$. Recall that the operator 
$$
R(u):=\nabla(DH)(u)=\nabla^2h(u)+h(u)\id_{u^\perp}
$$
on $u^\perp$ is defined for a.e. $u\in \Stwo$. Moreover,  its smaller  $r_{\text{min}}(u)$ and larger $ r_{\text{max}}(u)$  eigenvalues
satisfy 
$$
0\le r_{\text{min}}(u)\le r_{\text{max}}(u)\le 1
$$
for a.e. $u\in \Stwo$. We then may reformulate the conjecture given in the introduction more precisely as follows. 
\\\\
{\bf Conjecture}.  {\it $K$ is extreme if and only if  
\be\label{extremeHcond}
r_{\text{min}}(u)=0\;\text{ or }\; 
r_{\text{max}}(u)=1
\ee
for a.e. $u\in \Stwo$. }
\\
\par We will check that Meisnner polyhedra and rotated Reuleaux polygons satisfy the condition above. To this end, we will use \eqref{RplusRminusOne}
\be\label{ResnickPreCond}
R(u)+R(-u)=\id_{u^\perp} \text{  for a.e. $u\in \Stwo$},
\ee
which also implies  
\be\label{rminrmaxCond}
r_{\text{min}}(u)+r_{\text{max}}(-u)=1
\ee
for a.e. $u\in \Stwo$. 
\begin{figure}[h]
     \centering
         \includegraphics[width=.8\textwidth]{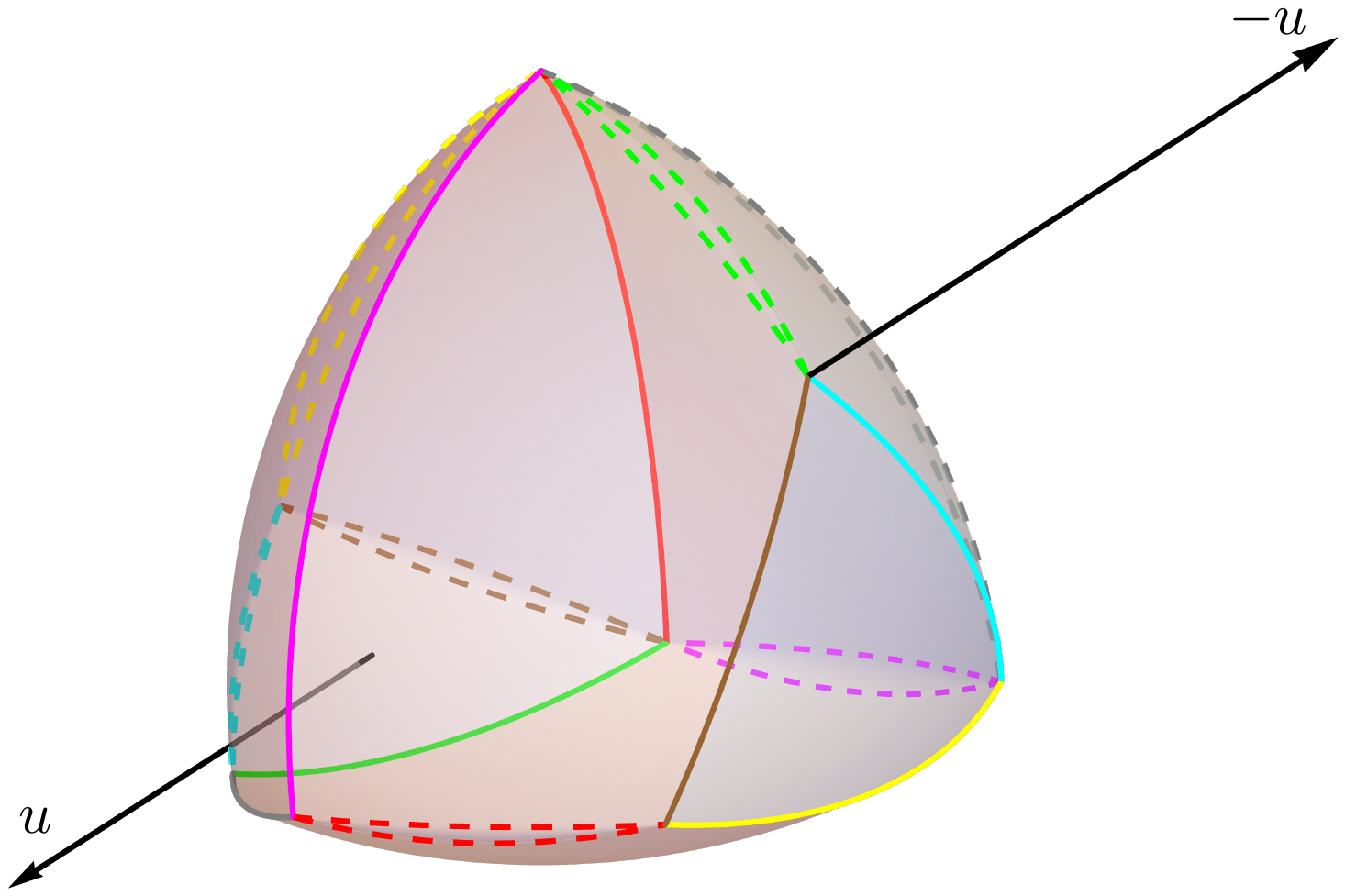}
         \caption{This is a Meissner polyhedron $K$ with two outward unit normals $u$ and $-u$. Note that $r_{\text{min}}(u)=r_{\text{max}}(u)=1$, as $u$ is normal to $\partial K$ on a portion of a unit sphere. And observe $r_{\text{min}}(-u)=r_{\text{max}}(-u)=0$, since $-u$ is normal to $\partial K$ at a vertex.}
         \label{MeissConjFig}
\end{figure}
\begin{ex}
Suppose $K$ is a Meissner polyhedron with support function $H$.  Observe that $\partial K$ is piecewise smooth and is a union finitely many portions of spheres of radius one and sections of spindle tori obtained by rotating a circular arc of radius one about a line in $\R^3$.  

\par Let us suppose that $S\subset \Stwo$  is open and each $u\in S$ is  normal to $\partial K$ on one of its spherical parts.  Then 
$H(u)=|u|+x\cdot u$ for $u\in S$,  where $x$ is the center of the sphere in question.  In this case,  
$
R(u)=\id_{u^\perp}
$
for each $u\in S$.  Also note each $u\in -S$, $H(u)=x\cdot u$ so $R(u)$ is the 0 operator on $u^\perp$. See Figure \ref{MeissConjFig}.

\par Now assume that $A\subset \Stwo$  is open and for each $u\in A$ is normal to $\partial K$ on one of its spindle portions. By Corollary \ref{RSpindleTorus}
in the appendix, $r_{\text{max}}(u)=1$ and $0<r_{\text{min}}(u)<1$ for each $u\in A$. Further, we note that $-A$ corresponds to the normals on a circular edge of $\partial K$ opposite the spindle. In this case $r_{\text{min}}(u)=0$ and $0<r_{\text{max}}(u)<1$ by \eqref{rminrmaxCond}. See Figure \ref{MeissConjFig2}. We conclude that condition stated in the conjecture holds for Meissner polyhedra.  
\end{ex}
\begin{figure}[h]
     \centering
         \includegraphics[width=.8\textwidth]{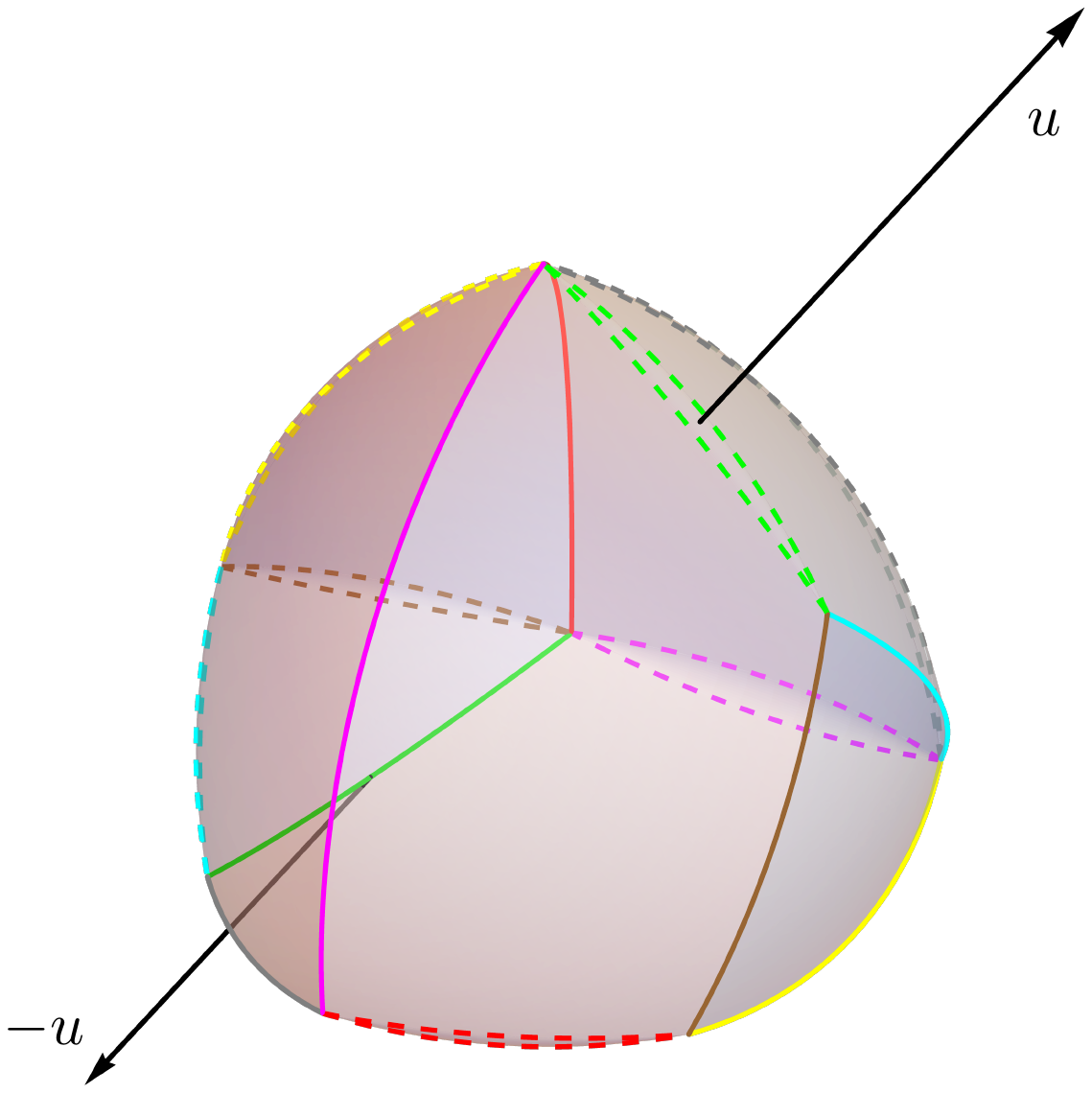}
         \caption{This is the same Meissner polyhedron $K$ in Figure \ref{MeissConjFig} again shown with two outward unit normals $u$ and $-u$. Here $0<r_{\text{min}}(u)<1$ and $r_{\text{max}}(u)=1$, as $u$ is normal to $\partial K$ on a segment of a spindle torus. Also notice that $r_{\text{min}}(-u)=0$ and $0<r_{\text{max}}(-u)<1$, since $-u$ is normal to $\partial K$ along a circular arc with radius less than one. }
         \label{MeissConjFig2}
\end{figure}
\begin{figure}[h]
     \centering
         \includegraphics[width=.8\textwidth]{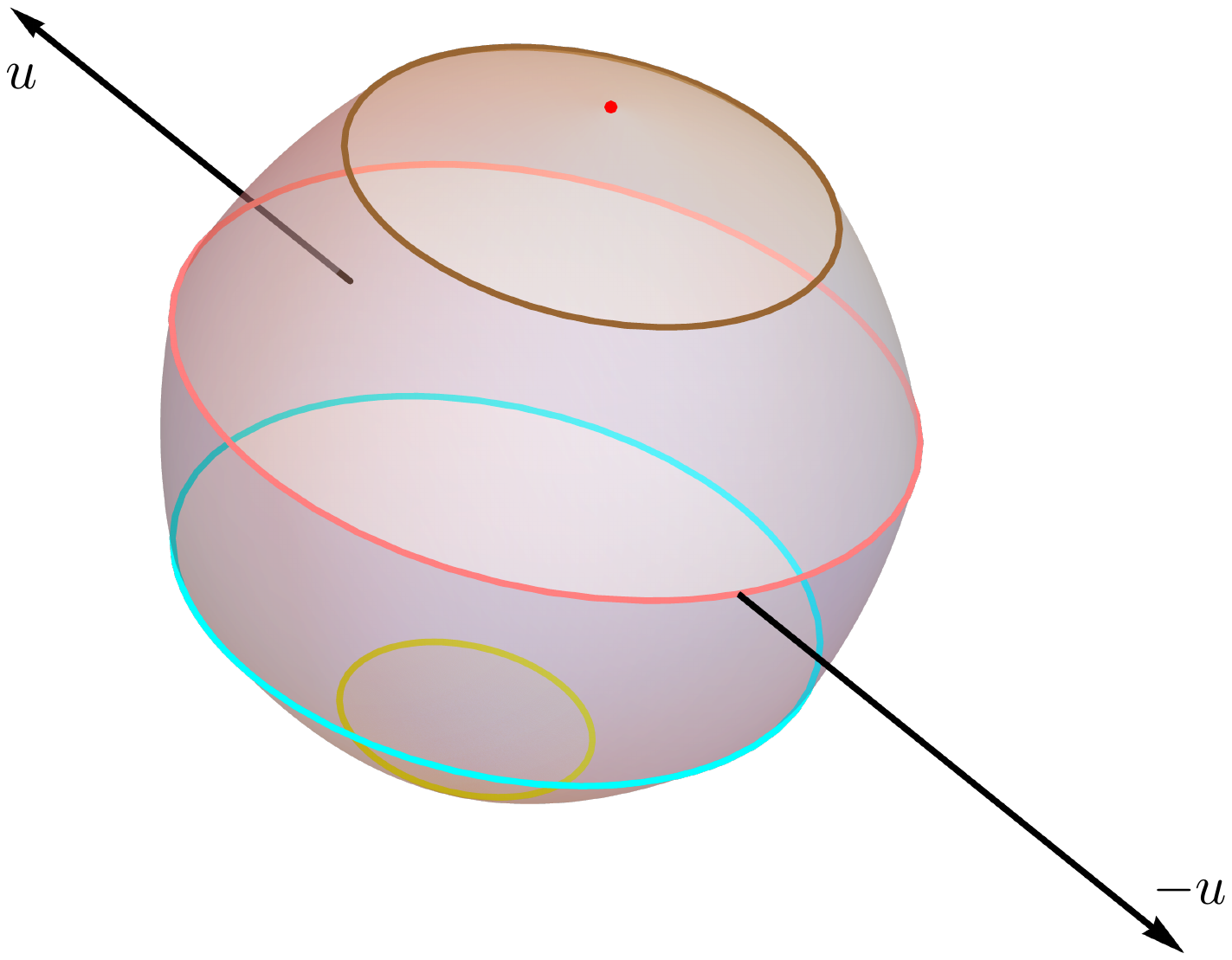}
         \caption{This is a rotated Reuleaux polygon $K$ with two outward unit normals $u$ and $-u$. Note that $0<r_{\text{min}}(u)<1$ and $r_{\text{max}}(u)=1$, since $u$ is normal to $\partial K$ on a portion of a spindle torus.  And $r_{\text{min}}(-u)=0$ and $0<r_{\text{max}}(-u)<1$, as $-u$ is normal to $\partial K$ along a circular arc with radius less than one. }
         \label{RotReulConj}         
\end{figure}
\begin{ex}
Suppose $K\subset \R^3$ is a constant width shape which is axially symmetric about the $x_3$-axis.  In spherical coordinates $(\theta,\psi)$, the matrix representation of $R$ is 
$$
\left(
\begin{array}{cc}
\displaystyle\frac{h'}{\tan\phi} +h &\displaystyle 0  \\
\displaystyle0 & h''+h
\end{array}
\right)
$$
for all $\theta\in [0,2\pi]$ and almost every $\phi\in [0,\pi].$
Here $h(\phi)=H(u(\theta,\phi))$ is independent of $\theta$ since $K$ is axially symmetric. Theorem \ref{thmC} asserts that $K$ is extreme if and only if the shape that generates it is extreme. This happens in turn if and only if $h(\phi)''+h(\phi)\in \{0,1\}$ for almost every $\phi\in [0,\pi]$ by Kallay's theorem. In this case, one of the eigenvalues of the above matrix is equal to $0$ or $1$ for almost every $\phi\in [0,\pi]$ and all $\theta\in [0,2\pi]$. As a result, the stated condition of the conjecture holds. See Figure \ref{RotReulConj}. \end{ex}

\par The conjecture asserts that if there is a set of positive measure for which $0<r_{\text{min}}$ and $r_{\text{max}}<1$, then
$K$ is not extreme. We now give a more stringent sufficient condition for $K$ not to be extreme along these lines. 
\begin{prop}
Suppose $A\subset\Stwo$ is open and nonempty, $\delta\in (0,1/2)$, and 
\be\label{extremeHcond}
\delta\le r_{\text{min}}(u)\quad\text{and}\quad r_{\text{max}}(u)\le 1-\delta
\ee
for a.e. $u\in A$.  Then $K$ is not extreme. 
\end{prop}
\begin{proof}
It suffices to verify this assertion for $A$ and $-A$ disjoint, as we can reduce to this case by selecting an appropriate subset of $A$.  Choose $g_0\in C^\infty_c(A)$ which is not identically equal to $0$ and $t\in \R$ so that the eigenvalues of $t\left( \nabla^2g_0(u)+g_0(u)\id_{u^\perp}\right)$ belong to the interval $[-\delta,\delta]$ for all $u\in A$.  This can be accomplished since  $g$ and its derivatives are bounded. 

\par Next define 
$$
g(u)=
\begin{cases}
\;\; \;g_0(u), \quad &u\in A\\
-g_0(-u), \quad &u\in -A\\
\;\;\;\; 0,\quad &\text{otherwise}
\end{cases}
$$
for $u\in \Stwo$.  Note that $g$ is a smooth function. With our choice of $t$, the smaller eigenvalue of the operator
\be\label{htgOperator}
L(u)=\nabla^2h(u)+h(u)\id_{u^\perp}+t\left( \nabla^2g(u)+g(u)\id_{u^\perp}\right)
\ee
is at least $r_{\text{min}}(u)-\delta\ge 0$ for a.e. $u\in A$.   Likewise, its larger eigenvalue is at most  $r_{\text{max}}(u)+\delta\le 1$ a.e. $u\in A$.  By \eqref{ResnickPreCond} and recalling that $g$ is antisymmetric, $L(u)+L(-u)=\id_{u^\perp}$ for almost every $u\in - A$. It follows that the eigenvalues of $L(u)$ also belong to $[0,1]$ for a.e. $u\in -A$.   These 
eigenvalue bounds also clearly hold for $u\not\in A\cup (-A)$.

\par By Remark \ref{RecoveryRemark}, the homogenous extensions of $h\pm tg$ are the support functions
of constant width bodies $K_\pm\subset \R^3$.  Since 
$h=\frac{1}{2}(h+tg)+\frac{1}{2}(h-tg)$,
$$
K=\frac{1}{2}K_++\frac{1}{2}K_-.
$$
And as the homogenous extension of $g$ is not equal to a linear function, $K_+$ and $K_-$ are not translates of one another. Therefore, $K$ is not extreme.
\end{proof}
 For a given open $A\subset \Stwo$, we'll say $DH(A)\subset \partial K$ is {\it smooth} if $DH$ is smooth and $r_\text{min}$ is positive on $A$.  In this scenario, the inverse function theorem implies that $DH$ is invertible in a neighborhood of each point of $A$, so the Gauss map is well-defined on $DH(A)$.   
 
 \par The subsequent assertion is implicitly contained in Theorem 5 of \cite{MR2342202}, where the Bayen, Lachand-Robert, and Oudet proved that for any volume-minimizing constant width body $K$ and open $A\subset K$, either $DH(A)$ or $DH(-A)$ is not smooth.  We note that the equality condition of the Brunn-Minkowski inequality can be used to show that any volume-minimizing constant width shape is necessarily extreme. 
\begin{cor}
If $K$ is extreme, for each nonempty open $A\subset \Stwo$ either $DH(A)$ or $DH(-A)$ is not a smooth subset of $\partial K$. 
\end{cor}
\begin{proof}
We will argue by contradiction and suppose that $DH(A)$ and $DH(-A)$ are smooth.  Then  $r_\text{min}(u)>0$ for each $u\in A,-A$. Let $B $ be a nonempty open subset of $\Stwo$ whose closure $\overline{B}\subset A$. As $DH$ is continuous differentiable, there is some $\delta>0$ small enough for which 
$r_\text{min}(u)\ge\delta$ for $u\in B,-B$.  In view of \eqref{rminrmaxCond}, 
$$
r_\text{max}(u)=1-r_\text{min}(-u)\le 1-\delta
$$
for $u\in B$. By the previous lemma, $K$ is not extreme, which contradicts our hypothesis. 
\end{proof}
We conclude this note with the following result, which was discovered by Anciaux and Guilfoyle as a necessary condition for a volume-minimizing constant width shape \cite{MR2763770}.   
\begin{cor}
Suppose $K$ is extreme and $DH(A)$ is smooth for some nonempty open $A\subset \Stwo$. Then 
$r_\text{max}(u)=1$ for each $u\in A$.
\end{cor}
\begin{proof}
By the previous corollary, $DH(-A)$ is not smooth. Nevertheless 
$$
DH(u)+DH(-u)=u
$$
for all $u\in \Stwo$ by \eqref{Hconstantwidthcondition}. Therefore, $DH$ is smooth on $-A$, since it is smooth on $A$.  We claim that 
\be\label{rminClaim}
\text{$r_\text{min}(u)=0$ for $u\in -A$.}
\ee
If $r_\text{min}(u_0)>0$ for some $u_0\in -A$, then there is a neighborhood $B$ with $u_0\in B\subset -A$ and $r_\text{min}(u)>0$ for each $u\in B$.  This would imply that $DH(B)$ and $DH(-B)\subset DH(A)$ are smooth. However, this would contradict the previous corollary.  We conclude \eqref{rminClaim}. The conclusion that $r_\text{max}(u)=1$ for $u\in A$ now follows from \eqref{rminrmaxCond}.
\end{proof}

\appendix 

\section{Spindle torus computations}
Fix $0<a< 1$. The circular arc 
$$
(x_1+\sqrt{1-a^2})^2+x_3^2=1, \quad x_1\ge 0
$$
in the $x_1x_3$-plane joins $ae_3$ and $-ae_3$. If we rotate this arc about the $x_3$-axis, we obtain the inner surface $\textup{T}_a$ of a spindle torus. See Figure \ref{SpindleFigure}. The surface $\textup{T}_a$ is described by the equation 
\be\label{SpindleSurface}
\left(\sqrt{x_1^2+x_{2}^2}+
\sqrt{1-a^2}\right)^2+x_3^2= 1.
\ee
It is easy to check that $\text{T}_a$ bounds a strictly convex body.  
\begin{rem}
Recall the region $ B(c_y)$ defined by \eqref{BcxSpindleFormula}, where $y=(y_1,y_3)$ is vertex of a Reuleaux polygon in $x_1x_3$-plane which is symmetric with respect to the $x_3$-axis. Observe that the boundary of $B(c_y)$ is $T_{\sqrt{1-y_1^2}}+y_3e_3$.
\end{rem}

\par The surface $\textup{T}_a$ is smooth away from its endpoints $\pm a e_3$, 
and $|x_3|<a$ corresponds to a smooth point on this surface. In this case, the outward unit normal at $x$ is 
$$
u=\left(\left(\sqrt{x_1^2+x_{2}^2}+
\sqrt{1-a^2}\right)\frac{x_1}{\sqrt{x_1^2+x_2^2}},\left(\sqrt{x_1^2+x_{2}^2}+
\sqrt{1-a^2}\right)\frac{x_1}{\sqrt{x_1^2+x_2^2}},x_3 \right).
$$
Moreover, direct computation gives
$$
u\cdot x=1-\sqrt{1-a^2}\sqrt{u_1^2+u_2^2}.
$$
Unit normals with $u_3\ge a$ or $u_3\le -a$ will support $\text{T}_a$ at $ae_3$ or $-ae_3$, respectively. 
\begin{figure}[h]
     \centering
         \includegraphics[width=.7\textwidth]{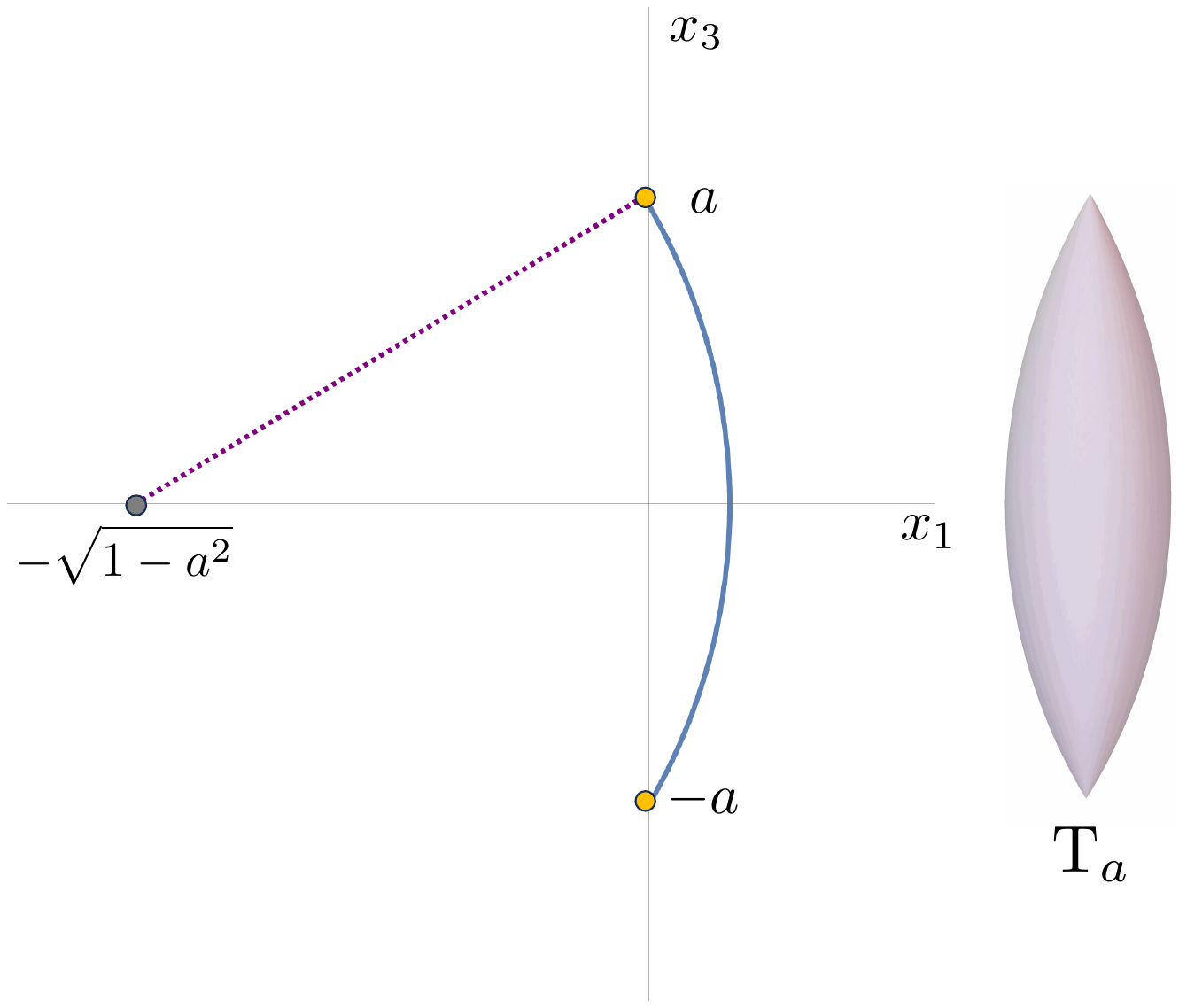}
         \caption{On the left is the generating curve for the surface described by the equation \eqref{SpindleSurface}. The corresponding surface of revolution  $\textup{T}_a$ is shown on the right. }\label{SpindleFigure}
\end{figure}
\par We summarize the observations made above in the following proposition. 
\begin{prop}\label{SuppFunctionSpa}
The support function of the region bounded by $\textup{T}_a$ is 
$$
H(u)
=\begin{cases}
au_3, \quad u_3/|u|\ge a\\\\
|u|-\sqrt{1-a^2}\sqrt{u_1^2+u_2^2}, \quad -a\le u_3/|u|\le a\\\\
-au_3, \quad u_3/|u|\le -a
\end{cases}
$$
for $u\neq 0$.
\end{prop}
The corollary below implies that the maximum principal radius of curvature of $\textup{T}_a$ is equal to one and
the minimum principal radius of curvature is positive and less than one.  This is also true for the inner part of 
any spindle torus in $\R^3$. 
\begin{cor}\label{RSpindleTorus}
Suppose $u\in \Stwo$ with $-a<u_3<a$ and set $R(u)=\nabla(DH)(u)$. Then 
$$
R(u)(e_3-u_3u)=e_3-u_3u
$$
and
$$
R(u)u\times e_3=\left(1-\frac{\sqrt{1-a^2}}{\sqrt{1-u_3^2}}\right)u\times e_3.
$$
\end{cor}
\begin{proof}
We will use the notation $P_{b^\perp}$ to denote orthogonal projection onto $b^\perp$ for $b\neq 0$. The previous proposition implies 
$$
H(v)=|v|-\sqrt{1-a^2}|P_{e_3^\perp}v|
$$
when $-a<v_3/|v|<a$. Direct computation of the gradient of $DH$ at $v=u$ gives
$$
R(u)=\id_{u^\perp}-\sqrt{1-a^2}\frac{P_{P_{e_{3}^\perp}u}}{|P_{e_3^\perp}u|}P_{e_3^\perp}.
$$

\par Since $P_{u^\perp} e_3=e_3-(u\cdot e_3)u\in u^\perp$ and 
$$
P_{P_{e_3^\perp}u}P_{e_3^\perp}(P_{u^\perp} e_3)=-(u\cdot e_3)P_{P_{e_3^\perp}u}P_{e_3^\perp}u=0,
$$
it follows that $R(u)P_{u^\perp} e_3=P_{u^\perp} e_3$.  Also note that as $u\times e_3$ is orthogonal to both $u$ and $e_3$, it is also orthogonal to $P_{u^\perp} e_3$. Therefore, 
$$
P_{P_{e_3^\perp}u}P_{e_3^\perp}(u\times e_3)=P_{P_{e_3^\perp}u}u\times e_3=u\times e_3,
$$  
and
$$
R(u)(u\times e_3)=\left(1-\frac{\sqrt{1-a^2}}{|P_{e_3^\perp}u|}\right)u\times e_3.
$$
Finally, we note $|P_{e_3^\perp}u|=\sqrt{u_1^2+u_2^2}=\sqrt{1-u_3^2}$.
\end{proof}

\bibliography{ESbib}{}

\bibliographystyle{plain}

\typeout{get arXiv to do 4 passes: Label(s) may have changed. Rerun}

\end{document}